\documentclass[12pt]{article}
\usepackage{amssymb,amsfonts,amsmath,amsthm}
\usepackage{indentfirst}
\usepackage{geometry}
\newtheorem{theorem}{Theorem}[section]
\newtheorem{corollary}[theorem]{Corollary}
\newtheorem{proposition}[theorem]{Proposition}
\newtheorem{lemma}[theorem]{Lemma}
\newtheorem{remark}[theorem]{Remark}
\newtheorem{example}[theorem]{Example}
\newtheorem{definition}[theorem]{Definition}

\numberwithin{equation}{section}

\def\T{\mathbb{T}}

\def\R{\mathbb{R}}
\def\Z{\mathbb{Z}}
\def\N{\mathbb{N}}
\def\C{\mathbb{C}}
\def\Q{\mathbb{Q}}

\newcommand{\supp}{\mbox{supp}\,}
\newcommand{\spec}{\textit{spec}\,}
\def\D{\mathcal{D}'}
\def\O{\mathcal{O}}
\def\H{\mathcal{H}}
\def\veps{\varepsilon}

\def\l{\ell}
\def\to{\rightarrow}
\def\To{\longrightarrow}
\def\dim{\mbox{dim}}
\def\diag{\mbox{diag}}
\def\ds{\displaystyle}

\title{Global Hypoellipticity for First-Order Operators on Closed Smooth Manifolds}

\author{
Fernando de \'Avila Silva \thanks{PPGM - UFPR - Brazil (Supported by CAPES Foundation,  Ministry of Education of Brazil) }
\and  Todor Gramchev \thanks{DMI - UNICA - Italy (Partially supported by a PRIN project of MIUR, Italy and GNAMPA, INDAM)}
\and Alexandre Kirilov \thanks{DMAT - UFPR - Brazil (Partially supported by IMI--UFPR and CAPES Foundation)} \thanks{Corresponding Author}
}

\begin{document}

\maketitle


\begin{abstract}
The main goal of this paper is to address global hypoellipticity issues for the class of first order pseudo-differential operators $L = D_t + C(t,x,D_x)$, where $(t,x) \in \T \times M$, $\T$ is the one-dimensional torus, $M$ is a closed manifold and $C(t,x,D_x)$ is a first order pseudo-differential operator on $M$, smoothly depending on the periodic variable $t$. In the case of separation of variables, when $C(t,x,D_x) = a(t)p(x,D_x)+ib(t)q(x,D_x)$, we give necessary and sufficient conditions for the global hypoellipticity of $L$. In particular, we show that the famous \emph{(P)} condition of Nirenberg-Treves is neither necessary nor sufficient to guarantee the global hypoellipticity of $L$.
\end{abstract}


\

  \begin{center}
  \begin{minipage}{\textwidth}
    \tableofcontents
    \addtocontents{toc}{~\hfill\textbf{Page}\par}
  \end{minipage}
  \end{center}


\section{Introduction}

The main goal on the present paper is to investigate the Global Hypoellipticity (GH) of the following class of operators:
\begin{equation}\label{Super-General-L}
L \doteq D_t + C(t,x,D_x), (t,x) \in \T \times M,
\end{equation}
where $\T=\R/2\pi\Z$ stands for the flat torus, $M$ is a closed smooth manifold (compact smooth  manifold without boundary) endowed with a
positive measure $dx$ and $C(t,x,D_x)$ is a first order pseudo-differential operator on $M$ smoothly depending on the periodic variable $t$.

We propose a novel approach, as far as we know, based on generalizations for parameter depending operators which were inspired by:
\begin{itemize}
   \item [ i)] By J. Hounie's abstract approach (Trans AMS, 1979) for the study of Global Hypoellipticity and Global Solvability of the abstract operator $\partial_t + b(t,A) = \partial_t + b_0(t) A + r(t,A)$, with $t\in \T$ and $A$ being a linear self-adjoint operator, densely defined in a separable complex Hilbert space $H$ which is unbounded, positive, and has eigenvalues diverging to $+\infty$; and $r(t,A)$ is a lower order term in a suitable sense.

    \item [ ii)]  Secondly, we mention the S. Greenfield's and N. Wallach's paper (Trans AMS, 1973) where the authors investigate the Global Hypoellipticity of invariant differential operators with respect to the eigenspaces of a fixed elliptic normal differential operator $E$, as well as the recent generalization of the notion of invariance for elliptic pseudo-differential operators on compact manifolds by J. Delgado and M. Ruzhansky (C.R. Math. Acad. Sci.,2014) where the authors use a discretization approach based on the Fourier expansions for characterizing functional spaces defined by R.T. Seeley, (Proc. AMS, 1965 and 1969). We emphasize that the novelty in our case is the presence of parameter $t$, which might lead to bifurcation type behavior in the presence of multiple eigenvalues.

    \item [ iii)] Finally, we apply  reduction to normal forms for first order operators on tori cf. D. Dickinson, T. Gramchev and M. Yoshino (Proc. Edinb. Math. Soc., 2002) in our abstract context. Here comes on of the main novelties of our paper: we introduce the notion of Diophantine sequences which turns out to be necessary and sufficient condition for the Global Hypoellipticity provided the imaginary part mean values $\nu_j =\int_{0}^{2\pi} B_j(t) dt$ growth at most logarithmically.
\end{itemize}

We observe that, for such Diophantine sequences, not surprisingly in view of the appearance of sequences which are not integers, the celebrated theorems in Diophantine metric theory are not applicable to our abstract Diophantine condition, see V. Beresnevich, D. Dickinson and S. Velanu, with an appendix by R. C. Vaughan \cite{BDV07} for general results on exceptional sets, and A. Gorodnik and A. Nevo \cite{GN15}.

We emphasize that the study of the (GH) of such a general class of operators is a highly non-trivial problem and it seems impossible to attack it by a unified approach, even when $C(t,x,D_x)$ is a first order differential operator on $M$. We mention that the main body of results on (GH) for differential operators is localized in the case where the compact manifold $M$ is a flat torus, see the impressive list of works \cite{BERG99,BCP04,BDGK15,CC00,DGY97,GW1,HOU79,HOU82,Petr11} and the references therein.

However, we remark that it is possible to treat a general class of operators without requiring smoothness with respect to the dual variables, as in the recent works of M. Ruzhansky and V. Turunen, see \cite{RT3,RT1,RT2}.

Moreover, we recall the famous Katok conjecture in \cite{Kat01,Kat03}, and also, the related Greenfield-Wallach conjecture in \cite{GW3} which states: if a closed, connected, orientable manifold admits a (GH) vector field, then this manifold is diffeomorphic to a torus and this vector field is smoothly conjugate to a constant Diophantine vector field (see also G. Forni \cite{Fo08} and A. Kocsard \cite{Koc09}).

Our crucial hypothesis is inspired by two works: S. Greenfield's and N. Wallach's paper \cite{GW2} in which the authors investigate the (GH) of invariant differential operators with respect to the eigenspaces of a fixed elliptic normal differential operator $E$, and  J. Delgado's and M. Ruzhansky's papers \cite{DR15,DR14,DR14JFA} in which they investigate the recent generalization of the notion of invariance for elliptic pseudo-differential operators on compact manifolds. Both notions lead to the possibility of using discretization approaches based on the Fourier expansions for characterizing functional spaces, defined by R.T. Seeley \cite{See65,See69}.

Our starting point is analogous: we fix an elliptic pseudo-differential operator $E(x,D_x)$ on $M$, and assume
\begin{equation}\label{hyp1}
[C(t,x,D_x),E(x,D_x)]=0, \ \ \forall t \in \T.
\end{equation}

However, in view of the presence of a global bifurcation of the parameter $t$, the commutation hypothesis \eqref{hyp1} is not sufficient. We also had to assume that
\begin{equation}\label{hyp1x}
C(t,x,D_x) \mbox{ is normal, namely, \ } C^*C=C \, C^*,
\end{equation}
where $C^*$ stands for the adjoint of $C$ with respect to $L^2(M,dx)$.

In fact, writing  uniquely
\begin{equation}\label{hyp1-AeB}
  A = \frac{C+C^*}{2} \ \mbox{ \ and \ } \ \ B = \frac{C-C^*}{2i},
\end{equation}
we have
\begin{equation}\label{hyp2}
 C(t,x,D_x) = A(t,x,D_x)+iB(t,x,D_x),
\end{equation}
and these two assumptions (commutativity with $E(x,D_x)$ and normality) are equivalent to the following commutative conditions:
\smallskip
\begin{equation}\label{hyp-self}
A^*(t,x,D_x) = A(t,x,D_x) \ \ \textrm{ and } \ \ B^*(t,x,D_x) = B(t,x,D_x);
\end{equation}
\smallskip
\begin{equation}\label{hyp3}
[A(t,x,D_x),E(x,D_x)]= [B(t,x,D_x),E(x,D_x)]=0;
\end{equation}
\smallskip
\begin{equation}\label{hyp4}
[A(t,x,D_x),B(t,x,D_x)]=0;
\end{equation}
for any $t \in \T$.

We recall the reader that, even for matrices, the centralizer is a not commutative group; thus the condition \eqref{hyp3} does not imply \eqref{hyp4}.

One more restriction is necessary, because of the possible bifurcation phenomena with respect to $t$, when we have multiple eigenvalues divergent to infinity:
\begin{quote}
{\em
 there exist unitary operators $S_t:L^2(M)\rightarrow L^2(M), \ S^*_t=S_t^{-1},$ smoothly depending on $t$, such that
 \begin{equation*}\label{hyp4-5}
    S_t^* A \, S_t \ \mbox{ \ and \ } \ S_t^* B \, S_t
  \end{equation*}
are simultaneously diagonal on the eigenspaces of  $E(x,D_x)$.}
\end{quote}

In the particular case of the separation of variables we can write \eqref{hyp1-AeB} as
\begin{align*}
 A(t,x,D_x) & = a(t)\otimes p(x,D_x) \ \mbox{ and } \\[2mm]
 B(t,x,D_x) & = b(t) \otimes q(x,D_x),
\end{align*}
where $a$ and $b$  are smooth, real functions on $\T$, and $p(x,D_x)$ and $q(x,D_x)$ are first order pseudo-differential operators on $M$, and the hypothesis \eqref{hyp-self}, \eqref{hyp3} and \eqref{hyp4} are respectively equivalent to
\begin{equation}\label{hyp-selfpq}
p(x,D_x)=p^*(x,D_x) \ \textrm{ and } \ q(x,D_x) = q^*(x,D_x);
\end{equation}
\begin{equation}\label{hyp3pq}
[p(x,D_x),E(x,D_x)] =0 \ \textrm{ and } \ [q(x,D_x),E(x,D_x)]=0;
\end{equation}
\begin{equation}\label{hyp4pq}
[p(x,D_x),q(x,D_x)]=0.
\end{equation}

Observe that, from these hypothesis, we obtain that $p(x,D_x)$ and $q(x,D_x)$ are simultaneously diagonalizable on each eigenspace of $E(x,D_x)$, therefore there exists a unitary operator $S$ such that
\begin{equation}\label{new-hyp4-5}
    S^* p(x,D_x) S \ \mbox{ and } \ S^* q(x,D_x) S
\end{equation}
are diagonal operators.

We outline our main novelties in the separation of variables case. First, assuming $b \not \equiv 0$ and denoting by $\{\nu_j\}$ the sequence of elements of the diagonal of $q(x,D_x)$, given in \eqref{new-hyp4-5}, and assuming that $|\nu_j| \to \infty$, we have:
\begin{enumerate}
  \item[i.] if $b$ does not change, then $L$ is (GH);
  \item[ii.] if $b$ changes sign and the growth of $|\nu_j|$ is super-logarithmic, then $L$ is not (GH);
  \item[iii.] finally, if $b$ changes sign and the growth of $|\nu_j|$ is at most logarithmic, then $L$ is (GH) if, and only if, a Diophantine phenomena occur.
\end{enumerate}

In the case where $b \equiv 0$, we denote by $\{\mu_j\}$ the sequence of the diagonal of $p(x,D_x)$ on the eigenspaces of $E(x,D_x)$, and by $a_0$ the mean value of $a(t)$ on $\T$. Then we have that $L$ is (GH) if, and only if, $a_0$ has at most finitely many resonances with respect to the sequence $\{\mu_j\}$, namely there exists $j_0\in \N$ such that
\begin{equation*}
a_0\mu_j \not\in \Z, \ \forall j \geqslant j_0,
\end{equation*}
and $a_0$ is non-Liouville with respect to the sequence $\{\mu_j\}$, namely there exist $\delta \geqslant 0$, $C>0$ and $R\gg 1$ such that \begin{equation}\label{nonLiouville-introd}
\inf_{\l \in \Z} |a_0\mu_j + \l| \geqslant C j^{-\delta}, \ \forall j \geqslant R.
\end{equation}

As an example, consider the following operators on the bidimensional torus:
\begin{eqnarray*}
  L   &=& D_t + a(t)D_x + ib(t)\log^\rho(2+|D_x|), \\[2mm]
  L_{a_0,b_0} &=& D_t + \ a_0\ D_x + \ ib_0\ \log^\rho(2+|D_x|),
\end{eqnarray*}
where $\rho>0, \ a,b \in C^\infty(\T)$, and $a_0$ and $b_0$ are the mean value of $a$ and $b$ on $\T$, respectively. In the case that $b \not \equiv 0,$ we have:

\begin{enumerate}
  \item if $\rho>1$, the operator $L$ is (GH) if, and only if, $b$ does not change sign;
  \item if $\rho \leqslant 1$, the operator $L$ is (GH) if, and only if, the operator $L_{a_0,b_0}$ is (GH), i.e. either $b_0\neq 0$ or $b_0=0$ and $a_0$ is an irrational non-Liouville number.
\end{enumerate}

We will discuss this example with more details on page \pageref{example-intro}, subsection \ref{example-intro}, where we compare our results with that obtained by J. Hounie, in \cite{HOU79}.

Observe that, in the presence of Diophantine phenomena, it is not possible to use the machinery of pseudo-differential calculus, since the inverse operator does not satisfy the difference estimates of M. Ruzhansky and V. Turunen \cite{RT3}, and J. Delgado and M. Ruzhansky \cite{DR14,DR14JFA}. On the other hand, we note that, in a different context, J. Delgado and M. Ruzhansky rely on invariant spaces without the presence of bifurcation parameters.

We point out that one of the crucial ingredients of our approach is the use of the corollary from Weyl's asymptotic counting function for $E$ on the asymptotic behaviour of the sequences $\{\mu_j\}$ and $\{\nu_j\}$, as well as the sequence space characterization by R. T. Seeley \cite{See65}. See also T. Gramchev, S. Pillipovic and L. Rodino \cite{GPR11} for hypoellipticity of Shubin type operators on $\R^n$.

\section{Functional spaces}

Let $M$ be a closed smooth manifold endowed with a positive measure $dx$. The inner product on the Hilbert space  $L^2(M)=L^2(M,dx)$ is  given by
\begin{equation*}
(f,g)_{L^2(M)} = \int_{M}{f(x)\overline{g(x)}dx}.
\end{equation*}

Denoting by $\H^s(M)$ the correspondent Sobolev space of order $s$ on $M$, we have
\begin{equation}\label{sobolev-dx}
C^\infty(M) = \bigcap_{s \in \R}\H^{s}(M) \mbox{ \ and \ } \D(M) = \bigcup_{s \in \R}\H^{s}(M).
\end{equation}

We denote by $\Psi^m(M)$ and by $\Psi^m_{cl}(M)$ the usual H\"ormander class of pseudo-differential operators and the classical pseudo-differential operators of order $m\in \R$, respectively (see, e.g., \cite{RT1}). Furthermore, we denote by $\Psi^m_{el}(M)$ the class of elliptic operators in $\Psi^m_{cl}(M)$, with $m>0$ in order to guarantee the discreteness of spectrum.

Suppose now that $E(x,D_x)\in \Psi^m_{el}(M)$, $m>0$,  is normal, namely
$$E(x,D_x)\circ E^*(x,D_x)=  E^*(x,D_x)\circ E(x,D_x).$$
Then:
\begin{enumerate}
  \item[$i.$] the spectrum $\spec(E)=\{\lambda_j; j\in \N\} \subset \R$ is discrete and $|\lambda_j| \to \infty$;
  \item[$ii.$] there is an orthonormal basis $\{\varphi_j\}_{j \in \N}$ for $L^2(M)$, where each $\varphi_j$ is a smooth function on $M$, such that $E\varphi_j = \lambda_j\varphi_j,$ for all $j \in \N$;
  \item[$iii.$] the eigenspace $E_{\lambda_j}$ of $E$ corresponding to $\lambda_j$, has finite dimension, for any $j\in \N$.
\end{enumerate}

We may assume that $\lambda_j>0,$ for any $j$. In fact, if this is not true, it will be enough to consider the following elliptic operator $E_{\delta} \doteq (E^*E + \delta)^{1/2}$, with $\delta>0$.

Finally, since $\lambda_j \to \infty$, counting the multiplicity of $\lambda_j$, we may assume that
\begin{equation}\label{hypsp1}
\spec (E)  = \{ 0 < \lambda_1 \leqslant \lambda_2 \leqslant\ldots \lambda_j \leqslant\ldots \to +\infty\},
\end{equation}
keeping the same orthonormal basis, after eventual reordering.

Now, fixed a normal elliptic operator $E$, as above, and an orthonormal basis $\{\psi^j_k\}_{k=1}^{d_j}$ of the eigenspace $E_{\lambda_j}$, we can write
\begin{equation*}
L^2(M) = \bigoplus_{j=1}^{\infty}E_{\lambda_j}, \ \ E_{\lambda_j} = \mbox{span } \left\{ \psi^j_{k} \right\}_{k=1}^{d_j}, \ j \in \N.
\end{equation*}
Thus, any distribution $u \in \D(M)$ may be represented as a Fourier series
\begin{equation}\label{Fourier1}
u(x) = \sum_{j\in \N}\left\langle  u^j, \psi^j(x) \right\rangle_{\C^{d_j}} = \sum_{j=1}^\infty\sum_{k=1}^{d_j} u_k^j \psi_k^j(x),
\end{equation}
where $u_k^j=\langle u,\overline{\psi_k^j} \ dx \rangle$. Of course, when $u \in L^2(M)$, we have the usual scalar product and $u_k^j=(u,{\psi_k^j})$.

Concerning the asymptotics of eigenvalues, from Shubin's theorems 15.2 and 16.1 in \cite{Shubin}, we have the following lemma.

\begin{lemma}[Weyl's Asymptotic Formula] \label{Weyl}
Let $E \in \Psi^m_{el}(M)$ be a normal elliptic operator with spectrum $\spec(E) = \{\lambda_j; j \in \N\}$. Then
\begin{equation}\label{Weyl-formula}
\lambda_j \sim \ c j^{\frac{m}{n}}, \quad j \to \infty,
\end{equation}
for some positive constant $c$.
\end{lemma}

Now, combining proposition 10.2 from \cite{Shubin} and Weyl's formula we have:

\begin{proposition}\label{prop-smooth-1}
For the series
\begin{equation}\label{1-prop-smooth-1}
\sum_{j \in \N} \sum_{k=1}^{d_j} c_k^{j} \psi_k^j(x),
\end{equation}
with complex coefficients $c_k^j$, the following three statements are equivalent:
\begin{enumerate}
        \item[i.] The series (\ref{1-prop-smooth-1}) converges in the $C^{\infty}(M)$ topology;
        \item[ii.] The series (\ref{1-prop-smooth-1}) is the Fourier expansion of some $f\in C^{\infty}(M)$;
        \item[iii.] For any integer $N$ we have
\begin{equation}\label{2-prop-smooth-1}
\sum_{j \in \N} |c_k^j|^2 j^{-N} < + \infty, \textit{ for each } k \in \{1, \ldots d_j\}.
\end{equation}
\hspace{-13.5mm} Moreover, the following conditions are equivalent:

        \item[iv.] The series (\ref{1-prop-smooth-1}) converges in the $\D(M)$ topology;
        \item[v.] The series (\ref{1-prop-smooth-1}) is the Fourier expansion of some $u\in \D(M)$;
        \item[vi.] For some integer $N,$ (\ref{2-prop-smooth-1}) holds;
\end{enumerate}
\end{proposition}

And we have the following characterization of Sobolev spaces.

\begin{proposition}\label{Sobolevmatr}
Let $\Gamma = \{\Gamma^j \}_{j=1}^\infty $ and $\Gamma^j \in M_{d_j\times d_j}(\C)$, with $j \in \N$.
For any $s\geqslant 0$ we have:
\begin{enumerate}
        \item[i.] The linear map $\Gamma: \H^s(M) \to \H^s(M)$ defined by
\begin{equation*}
\Gamma u = \sum_{j \in \N} \left\langle \Gamma^j  u^j, \psi^{j}(x)   \right\rangle_{\C^{d_j}}
\end{equation*}
is continuous if, and only if,
\begin{equation}\label{Hsnorm1}
\sup_{j \in \N} || (\Gamma^j)^* \Gamma^j \| = \sup_{j \in \N}
\left\{ \max \{ \sigma; \sigma \in \spec (\Gamma_j^* \Gamma_j )\}
\right\} < +\infty
\end{equation}
If the multiplicities are bounded, i.e.,
\begin{equation} \label{Hsnorm2}
\sup \{ d_j,  j \in \N \} = \overline{d} <+\infty,
\end{equation}
the condition \eqref{Hsnorm1} is equivalent to uniform boundedness of the entries of $\Gamma^j$, $j \in \N$, i.e.,
\begin{equation}\label{Hsnorm3}
\sup_{j \in \N}\ \max_{r,s\in\{1, \ldots,d_j\}} | \Gamma^j_{rs} | <+\infty.
\end{equation}
        \item[ii.] If $ 0 \notin \spec(E)$ and \eqref{Hsnorm3} holds, then the following expression is an equivalent norm on $\H^s(M)$

\begin{equation}\label{Hsnorm4}
||\Gamma u||_{\H^s(M)} \doteq || E^{s/m}\Gamma u||_{L^2(M)}
\end{equation}
if, and only if,

\begin{equation}\label{Hsnorm5}
\inf_{j \in \N} || (\Gamma^j)^* \Gamma_j \| = \inf_{j \in \N}
\left( \min  \{ \sigma; \sigma \in \spec (\Gamma_j^* \Gamma_j )\}\right) >0.
\end{equation}
\end{enumerate}
Moreover
\begin{equation} \label{Sobolev1matr}
u \in \H^s(M) \Longleftrightarrow \sum_{j\in\N} {\|u^j\|^2 \lambda_j^{\frac{2s}{m}} }< + \infty
              \Longleftrightarrow \sum_{j\in\N} {|u_j|^2 j^{\frac{2s}{n}} }< + \infty.
\end{equation}
\end{proposition} \vspace{3mm}

We also define the $x$-Fourier series of a distribution  $u \in \D(\T \times M)$
\begin{equation}\label{Fourier2matr}
u = \sum_{j=1}^\infty  \left\langle u^j(t), \psi^j(x)\right\rangle = \sum_{j=1}^\infty \sum_{k=1}^{d_j} u^j_k(t)\psi^j_k(x).
\end{equation}

\begin{proposition}\label{prop-smooth-2matr}
For the series
\begin{equation}\label{1-prop-smooth-2}
\sum_{j \in \N} \sum_{k=1}^{d_j} c_k^{j}(t) \psi_k^j(x),
\end{equation}
where $c_{j}^k \in C^{\infty}(\T)$, the three following statements are equivalent:
\begin{enumerate}
        \item[i.] The series \eqref{1-prop-smooth-2} converges in the $C^{\infty}(\T \times M)$ topology;

        \item[ii.] The series \eqref{1-prop-smooth-2} is the $x$-Fourier expansion of some $f\in C^{\infty}(\T \times M)$;

        \item[iii.] For any $k \in \N$ and any integer $N$,
             \begin{equation} \label{2-prop-smooth-2}
                \max_{t\in \T}|\partial_t^k c_{\ell}^j(t)| = \O(j^{-N}), \ \textrm{ as } j \to \infty,
             \end{equation}
   for each $\ell \in \{1, \ldots d_j\}$.
\end{enumerate}
   Moreover, the following conditions are equivalent:
\begin{enumerate}
        \item[iv.] The series (\ref{1-prop-smooth-2}) converges in the $\D(\T \times M)$ topology;

        \item[v.] The series (\ref{1-prop-smooth-2}) is the Fourier expansion of some $f\in \D(\T \times M)$;

        \item[vi.] For some real $N$, (\ref{2-prop-smooth-2}) holds.
\end{enumerate}
\end{proposition}

\section{(GH) and the separation of variables}

In the first part of this section, we are going to show that it is enough to consider the case where all eigenvalues of
$E(x,D_x)$ are simple. In the second part, we are going to give a more precise version of the theorem announced in the introduction.

\subsection{Reduction to the diagonal form \label{reduct-diag}}

Consider the following operator
\begin{equation}\label{Model L}
L \doteq D_t + a(t)p(x,D_x) + ib(t)q(x,D_x), \quad (t,x) \in \T\times M,
\end{equation}
where $a,b \in C^\infty(\mathbb{T})$, and $p(x,D_x)$ and $q(x,D_x)$ are self-adjoint pseudo-dif\-fer\-en\-tial operators, defined in $\Psi^1(M)$, which commute with a fixed normal elliptic operator $E(x,D_x)\in \Psi^1_{el}(M)$, namely
\begin{equation}\label{pq-commutes-with-E}
     [E, p(x,D_x)] = 0 \ \mbox{ and } \ [E, q(x,D_x)] = 0.
\end{equation}

We observe that the commutation conditions \eqref{pq-commutes-with-E}, \eqref{hyp1}, \eqref{hyp3} etc. guarantee that the considered operators are invariant with respect to $E$, with the notion of invariance introduced by J. Delgado and M. Ruzhansky in the recent paper \cite{DR15}.

Observe that, under this assumption, we have $p(E_{\lambda_j}) \subset E_{\lambda_j}$ and $q(E_{\lambda_j}) \subset E_{\lambda_j}$, for any $j \in \N$. In this case we say that the operators $p$ and $q$ are $E_{\lambda_j}$--invariants.

We can also rewrite the spectrum of $E(x,D_x)$ without counting the multiplicity, as in \cite{DR14JFA,DR14,GW1},
\begin{align}
         & \spec(E)  =  \{ 0 < \sigma_1<\sigma_2  <\ldots \sigma_j < \ldots \to +\infty\}, \label{hypsp2} \\[2mm]
         & \mbox{with \ mult} (\sigma_j) = d_j , \ j\in\ \mathbb{N}, \nonumber
\end{align}
and the corresponding orthonormal basis of $L^2(M)$ as
\begin{equation}\label{hypsp2a}
\{ e^j_k; \ k=1,2,\ldots,d_j, \ j\in\N \}.
\end{equation}

Thus, each eigenspace $E_{\sigma_j}$ has dimension $d_j$ and
\begin{equation}\label{basis-E-sigmaj}
E_{\sigma_j} = \mbox{span } \{ e^j_1, e^j_2, \ldots, e^j_{d_j} \}, \mbox{ for any } j \in \N.
\end{equation}

This way, any $u \in \D(\T \times M)$ can be represented by a $x$-Fourier series as follows:
\begin{equation*}
u = \sum_{j\in \N} \, \sum_{k=1}^{d_j}u^j_k (t) e^j_k (x).
\end{equation*}

Since $p(x,D_x)$ and $q(x,D_x)$ are also $E_{\sigma_j}$-invariants, we can consider the restrictions
\begin{equation}\label{restrict}
p_j(x,D_x): E_{\sigma_j} \To E_{\sigma_j}
\ \mbox{ and } \ \,
q_j(x,D_x): E_{\sigma_j} \To E_{\sigma_j}.
\end{equation}

Thus, for $u \in \D(\T \times M)$ we can write
\begin{align} \label{discrete2}
p(x,D_x)u &= \sum_{j\in \N} \left\langle P_j {U_j}(t), {e^j}(x)  \right\rangle_{\C^{d_j}}, \mbox{ and} \nonumber \\[2mm]
q(x,D_x)u &= \sum_{j\in \N} \left\langle Q_j {U_j}(t), {e^j}(x)  \right\rangle_{\C^{d_j}},
\end{align}
where $P_j$ and $Q_j$ are the complex self-adjoint $d_j \times d_j$ matrices of $p_j$ and $q_j$, with respect to the orthonormal basis of $E_{\sigma_j}$ given in \eqref{basis-E-sigmaj} and
\begin{equation}\label{system-Uj}
U_j(t)=\big(u_{k}^j(t)\big)_{d_j\times 1}, \ \mbox{ and } \ e^j=\big(e^{j}_k(x)\big)_{d_j\times 1}, \ \mbox{ for any } \ j \in \N.
\end{equation}

With this notation, $Lu = f$ is equivalent to the following sequence of differential equations:
\begin{equation}\label{system-1}
D_t {U_j}(t) + C_j(t) {U_j}(t) = {F_j}(t), \ j \in \N, \ t \in \T,
\end{equation}
where
\begin{equation}\label{system-C-Fj}
C_j(t) = a(t) P_j + i b(t) Q_j \ \mbox{ and } \ F_j(t) = \big(f_k^{j}(t)\big)_{d_j\times 1}.
\end{equation}

Now, we recall the following lemma of linear algebra.

\begin{lemma}\label{simult-diag.}
Let $\{T_{\alpha}:V \to V, \alpha \in \Lambda\}$ be a family of diagonalizable linear operators defined on a finite dimensional vector space $V$ such that
$[T_{\alpha}, T_{\beta} ] = 0, \forall \alpha, \beta \in \Lambda$.
Thus, there exists an ordered basis of $V$ in which any operator of this family has a diagonal representation. Moreover, if every $T_\alpha$ is normal, then there is an unitary matrix $S$ satisfying
\begin{equation*}
[T_{\alpha}] = S D_{T_\alpha} S^{*}, \ \forall \ {\alpha} \in \Lambda,
\end{equation*}
where  $D_{T_\alpha}$ is the diagonal matrix of the eigenvalues of $T_{\alpha}$.
\end{lemma}

For each $j \in \N$, the family $\{P_j,Q_j\}$ satisfies the hypothesis of lemma \ref{simult-diag.}, since $P_j^* = P_j$, $Q_j^* = Q_j$ and
$[p(x,D_x), q(x,D_x)] = 0$ is equivalent to
\begin{equation*}
  [P_j, Q_j ] = 0, \ \forall j \in \N.
\end{equation*}

Therefore, for each $j \in \N$, there exists an ordered basis of $E_{\sigma_j}$ such that
\begin{equation}\label{diag-simult-1}
S^*_j P_j S_j =   D_{P_j}  \ \textrm{ and } \ S^*_j Q_j S_j =   D_{Q_j},
\end{equation}
where each $S_j$ is a unitary matrix and
\begin{equation}\label{diag-simult-2}
D_{P_j} = \mbox{diag} \left ( \mu_{j,1},   \ldots,  \mu_{j,{d_j}} \right) \ \textrm{ and } \
D_{Q_j} = \mbox{diag} \left ( \nu_{j,1},   \ldots,  \nu_{j,{d_j}} \right)
\end{equation}
are diagonal matrices.

\smallskip
Now, if we write
\begin{equation}\label{V-G-of-diagonal-system}
{V_j}(t) \doteq S_j^* {U_j}(t) \ \textrm{ and } \ {G_j}(t) \doteq S_j^* {F_j}(t),
\end{equation}
where  ${U_j}(t)$ and ${F_j}(t)$ are defined in \eqref{system-Uj} and \eqref{system-C-Fj}, then we can rewrite the sequence of differential equations \eqref{system-1} as
\begin{equation}\label{diagonal-system}
D_t {V_j}(t) + C_j(t) {V_j}(t) = {G_j}(t), \ j \in \N,
\end{equation}
where $C_j(t) = a(t) D_{P_j} + i b(t)D_{Q_j}$.

Observe that the study of the behaviour of the solutions $U_j$ is equivalent to the study of $V_j$. Indeed, since $S_j$ is unitary we have
\begin{eqnarray*}
\left\| \partial^{k}_t {V_j}(t) \right\|^2_{\C^{d_j}} &=& \left\| S_j^*\cdot \partial^{k}_t {U_j}(t) \right\|^2_{\C^{d_j}} \\[2mm]
                              &=& \left\langle S_j\cdot S_j^*\partial^{k}_t {U_j}(t),\partial^{k}_t {U_j}(t) \right\rangle_{\C^{d_j}} \\[2mm]
                              &=& ||\partial^{k}_t {U_j}(t)||^2_{\C^{d_j}}.
\end{eqnarray*}

In particular, when $f \in C^{\infty}(\T \times M)$, the sequences $\{F_{j}(t)\}$ and $\{G_{j}(t)\}$ satisfy the condition
\eqref{2-prop-smooth-2}.

Then, the system \eqref{diagonal-system} is equivalent to the sequence of differential equations
\begin{equation}\label{entradas-eq}
D_t v_{\ell}^j(t) + c_{\ell}^j(t)v_{\ell}^j(t) = g_{\l}^j(t), \ j \in \N,
\end{equation}

\noindent with $c_{\ell}^j(t) = a(t) \mu_{\ell}^j + i b(t)\nu_{\ell}^j$, for each $\ell \in \{1, \ldots, d_j\}$.

If the solutions $V_j(t)$ of \eqref{diagonal-system} satisfy an estimate as
\begin{equation*}
\left\| \partial^{k}_t {V_j}(t) \right\|^2_{\C^{d_j}} \leqslant \ C j^N, \ j \to \infty,
\end{equation*}
then each $v_{\ell}^j(t)$ satisfies itself, and reciprocally.

It follows from this discussion that the global hypoellipticity of the operator $L$ is equivalent to the study of the solutions of the equations
\eqref{entradas-eq}. In this sense, it is enough to consider the case where the multiplicity of all eigenvalues is exactly equal to one
(simple eigenvalues).

\begin{remark} \label{obs-comuta}
To obtain the diagonal system \eqref{diagonal-system} we use lemma \ref{simult-diag.}, which requires the commutation $[P_j, Q_j] = 0$, for each $j\in \N$; thus in this point we stress the use of hypothesis
\begin{equation*}
[p(x,D_x), q(x,D_x)] =0.
\end{equation*}
\end{remark}

\subsection{(GH) for the diagonal form \label{GHT}}

We start by fixing a normal elliptic operator $E(x,D_x)\in \Psi^1_{el}(M)$, with spectrum
\begin{equation}\label{hypsp1bis}
\spec (E)  = \{ 0 < \lambda_1 \leqslant\lambda_2 \leqslant\ldots \lambda_j \leqslant\ldots \to +\infty\},
\end{equation}
where all eigenvalues $\lambda_j$ are simple, and the corresponding orthonormal basis is $\{\varphi_j\}_{j \in \N}$ for $L^2(M)$.
In this situation, all eigenspaces $E_{\lambda_j}$ have dimension 1.

\smallskip
Let $p(x,D_x),q(x,D_x) \in \Psi^1(M)$ be self-adjoint operators that commute with $E(x,D_x)$, namely
\begin{equation*}\label{pq-commutes-with-Ex}
     [E(x,D_x), p(x,D_x)] = 0 \ \ \mbox{ and } \ \ \ [E(x,D_x), q(x,D_x)] = 0,
\end{equation*}
and let $L$ be the operator
\begin{equation}\label{L-dim-1}
L \doteq D_t + a(t)p(x,D_x) + ib(t)q(x,D_x), \quad (t,x) \in \T\times M,
\end{equation}
with $a,b \in C^\infty(\mathbb{T})$ and set
\begin{equation}\label{a0-b0}
a_0 =  (2\pi)^{-1} \int_{0}^{2\pi}a(\tau) d\tau, \quad  b_0 = (2\pi)^{-1} \int_{0}^{2\pi}b(\tau) d\tau.
\end{equation}

\smallskip

Since $p(E_{\lambda_j}) \subset E_{\lambda_j}$, $q(E_{\lambda_j}) \subset E_{\lambda_j}$ and $\dim(E_{\lambda_j})=1,$ for any $j \in \N$,
there exist sequences of real numbers $\{\mu_j\}$ and $\{\nu_j\}$ such that
\begin{equation}\label{seqmunu}
p(x,D_x)\varphi_j = \mu_j \varphi_j \ \textrm{ and } \ q(x,D_x)\varphi_j = \nu_j \varphi_j, \ \ j \in \N.
\end{equation}

We point out that the behavior at the infinity of these sequences, play a decisive role in the study of the regularity of the operator $L$. \

\begin{definition}
We say that $a_0$ is non-Liouville with respect to the sequence $\{\mu_j\}$, if there exists $\delta \geqslant 0$, $C>0$ and $R\gg 1$ such that
\begin{equation}\label{nonLiouville}
\inf_{\l \in \Z} |a_0\mu_j + \l| \geqslant C j^{-\delta}, \ \forall j \geqslant R.
\end{equation}
If \eqref{nonLiouville} does not hold, we say that $a_0$ is Liouville with respect to the sequence $\{\mu_j\}$.
\end{definition}

\begin{definition}
The set of resonances of $a_0$ with respect to the sequence $\{\mu_j\}$ is defined by
\begin{equation}\label{nonres1}
\Gamma_{a_0} = \{ j \in \N; \ a_0 \mu_j \in \Z\}.
\end{equation}
\end{definition}

\begin{definition}\label{GH-definition}
   The operator $L$ is said to be globally hypoelliptic on $\T \times M$ (GH) if the conditions $u \in \D(\T \times M)$  and
$Lu \in {C}^\infty(\T \times M)$  imply $u \in C^\infty(\T \times M)$.
\end{definition}

Now we are in the position to enunciate our main results on the case of the separation of variables and simple eigenvalues.

\begin{theorem}\label{main-thm-dim-1}
Let $L$ be the operator defined in \eqref{L-dim-1}, and suppose that
\begin{equation*}
\lim_{j \to \infty} |\nu_j| = \infty.
\end{equation*}
Then:
  \begin{enumerate}\itemsep2pt
        \item[i.] if $b\equiv 0$, then $L$ is (GH) if, and only if, the resonance set $\Gamma_{a_0}$ is finite and $a_0$ is non-Liouville with respect to the sequence $\{\mu_j\}$;

        \item[ii.] if $b$ does not change sign and $b$ is not identical to zero, then $L$ is (GH);

    \item[iii.] if $b$ changes sign, then $L$ is not (GH), provided that there is a subsequence $\{\nu_{j_k}\}$ such that
       \begin{equation}\label{onboundhyp}
             \lim_{k \to \infty} \dfrac{|\nu_{j_k}|}{\log (j_k)} = + \infty.
       \end{equation}
  \end{enumerate}
However, if
   \begin{equation}\label{boundhyp}
                         \limsup_{j \to \infty} \dfrac{|\nu_{j}|}{\log (j)} = \kappa < + \infty,
   \end{equation}
we have:
\begin{center}
$L$ is (GH) \ if, and only if, \ $L_{a_0,b_0} \doteq D_t + a_0p(x,D_x) + ib_0 q(x,D_x)$ is (GH),
\end{center}
namely
\begin{enumerate}
                    \item[a.] if $b_0\neq 0$, then $L_{a_0,b_0}$ is (GH);

              \item[b.] if $b_0=0$, $L_{a_0,b_0}$ is (GH) if, and only if, $\Gamma_{a_0}$ is a finte set and $a_0$ is non-Liouville with respect to the sequence $\{\mu_j\}$.
\end{enumerate}
\end{theorem}

\subsection{Logarithmic growth and Sobolev spaces \label{example-intro}}

The purpose of this subsection is to establish a parallel between theorem \ref{main-thm-dim-1} above, and the results obtained by J. Hounie \cite{HOU79}, in the case where the growth of the sequence $\{\nu_j\}$ is at most logarithmic, that is, \eqref{boundhyp} holds.

We start by recalling the following example given in the introduction:
\begin{equation}\label{examplo-hounie-fail}
     L_\rho = D_t + a(t)D_x + ib(t)\log^\rho(2+|D_x|), \quad (t,x)\in \T\times\T,
\end{equation}
where $\rho>0$ and $a,b \in C^\infty(\T)$.

In the case that $b\not \equiv 0,$ our results imply that:
\begin{enumerate}
  \item[$i.$] if $\rho>1$, the operator $L_\rho$ is (GH) if, and only if, $b(t)$ does not change sign;
  \item[$ii.$] if $\rho \leqslant 1$, the operator $L_\rho$ is (GH) if, and only if, either $b_0\neq 0$, or $b_0=0$ and $a_0$ is an irrational non-Liouville number;
\end{enumerate}

Thus, the operator $L_\rho$ may be $C^\infty\!$--global hypoelliptic even when the function $b$ changes sign, that is, we are able to obtain examples in which the famous condition $(\mathcal{P})$ of Nierenberg-Treves fails, see \cite{NT70,NT71,T70}, and the first-order operator $L$ is $C^\infty$--global hypoelliptic.

We highlight that J. Hounie's abstract result in \cite{HOU79} could not be applied for the study of $C^\infty\!$--global hypoellipticity if
$a \equiv 0$ and $b \not \equiv 0$.

If $b \equiv 0$ our general result recaptures the theorem of J. Hounie on our example with $\H^0(\T)=L^2(\T)$, $\H^\infty(\T)=C^\infty(\T)$ and $\H^{-\infty}=\D(\T)$.

But, in the case $a\equiv 0$ and $b \not \equiv 0$  we have
\begin{equation}\label{example-hounie-fail2}
     \tilde{L}_\rho = D_t + ib(t)\log^\rho(2+|D_x|), \quad (t,x)\in \T\times\T,
\end{equation}
where $\rho>0$ and $b \in C^\infty(\T)$.

For this, consider the following self-adjoint pseudo-differential operator defined on the one-dimensional torus $\T$:
\begin{equation*}
  Q(x,D_x)= \log^\rho(2+|D_x|), \  \rho > 0,
\end{equation*}
and, following the ideas of J. Hounie, in \cite{CH77} and \cite{HOU79}, consider the scale of Sobolev spaces $\H^s_Q$ defined by $Q$, that is, each $\H^s_Q$ is the space of elements $u$ of $\D(\T)$ such that $Q^s u \in L^2(\T)$ or, equivalently,
\begin{equation*}
\H^s_Q = \{u \in \D(\T); \log^{\rho s}(2+|\xi|)\widehat{u}(\xi) \in \ell^2(\Z) \}, \ s \in \R.
\end{equation*}

We also denote
\begin{equation*}
\H^\infty_Q = \bigcap_{s\in \R}\H^s_Q \ \mbox{ \ and \ } \ \H^{-\infty}_Q = \bigcup_{s\in \R}\H^s_Q,
\end{equation*}
\noindent and
$$
  \H^\veps (\T) \mbox{ is the standard Sobolev space of order } \veps \in \R. \vspace{5mm}
$$

\begin{proposition}
$\H^\infty_Q(\T) \neq C^\infty(\T)$
\end{proposition}
\begin{proof}
We will show that for any $\veps >0$, $\H^\infty_Q(\T) \not \subset \H^{\veps} (\T)$.  Let  $\theta>1/2$ and set
\begin{equation*}
\psi(\xi) = |\xi|^{-1/2}\log^{-\theta}(|\xi|), \ \xi \in \Z, \ |\xi|\gg 1.
\end{equation*}

Note that

\begin{equation*}
\int_{ |\xi| \geqslant R}  \frac{1}{|\xi| \log^{2\theta } (|\xi|)} d\xi \sim
\int_R^{+\infty} \frac{1}{\rho\log^{2\theta } \rho} d\rho = \frac{1}{ (2\theta -1) \log^{2\theta -1 } (R) } < +\infty,
\end{equation*}

\noindent then $\{\psi (\xi)\}_{\xi \in \Z} \in  \ell^2(\Z)$.

Now, fix $\delta \in (0,1)$ and define
\begin{equation*}
\widehat{u}(\xi) = \left \{
\begin{array}{l}
\psi(\xi)e^{-\log^{\delta}(|\xi|) \log(\log(|\xi|))}, \ \textrm{ if } \xi \in \Z \setminus\{0\}, \\[3mm]
0, \ \ \textrm{ if } \xi=0.
\end{array} \right.
\end{equation*}

For  each $s>0$ we obtain
\begin{align}\label{ine-1}
\log^{s\rho} (|\xi|)\widehat{u} (\xi)  = & \psi (\xi) e^{\rho s \log (\log |\xi|)-\log ^\delta (|\xi| ) \log (\log |\xi|)) } \nonumber \\[2mm]
                                       = & \psi (\xi)e^{- (\log^\delta (|\xi| ) -\rho s) \log (\log |\xi|)  }, \quad |\xi| \gg 1.
\end{align}

For any $N>0$ there exists $R=R(N, \rho s)>0$, such that
\begin{equation*}
N < \log^{\delta}(|\xi|) - \rho s, \ |\xi| \geqslant R,
\end{equation*}
thus, $\forall |\xi| \geqslant R$
\begin{align*}
e^{-(\log^\delta (|\xi| ) - \rho s)  \log(\log |\xi|) }  \leqslant e^{- N  \log (\log |\xi|) } = (\log |\xi|)^{-N}.
\end{align*}

From \eqref{ine-1} we have
\begin{equation*}
\log^{s\rho} (|\xi|) \widehat{u} (\xi) \leqslant  \psi (\xi) (\log |\xi|)^{-N} \leq \psi (\xi), \ \ |\xi| \geq R,
\end{equation*}
and then $\{\widehat{u} (\xi)\}_{\xi \in \Z}$ defines a distriuition $u \in  \H^\infty_Q(\T)$.

Once $\delta <1$, it follows that
\begin{equation*}
\lim_{|\xi| \to \infty } \frac{\log^\delta (|\xi|) \log(\log |\xi|)}{\log (|\xi|)} = 0,
\end{equation*}
then for each $\varepsilon>0$:, there exist $R'>0$ such that
\begin{equation*}
\log^\delta (|\xi|) \log(\log |\xi|) \leqslant \varepsilon/2 \log (|\xi|), \ \forall |\xi| \geqslant R'.
\end{equation*}

Thus, for $|\xi|\geqslant R'$, we obtain
\begin{align*}
|\xi|^{\varepsilon}  \widehat{u}  (\xi) & =  \psi (\xi) |\xi|^\varepsilon e^{-\log ^\delta (|\xi|) \log (\log (|\xi|))) } \\[2mm]
        & =  \psi (\xi)e^{ \varepsilon \log (|\xi|) -\log ^\delta (|\xi|) \log (\log (|\xi|)) } \\[2mm]
        & \geqslant \psi (\xi)e^{ \varepsilon \log (|\xi|) - \varepsilon/2 \log (|\xi|)} \\[2mm]
        & =  \psi (\xi)e^{ \varepsilon/2 \log |\xi|} \\[2mm]
        & =   |\xi|^{-n/2 + \varepsilon/2} \log^{-\theta} (|\xi|).
\end{align*}

Since  $\{|\xi|^{-1/2 + \varepsilon/2} \log^{-\theta} (|\xi|)\}_{\xi \in \Z} \notin \ell^2(\Z)$, we obtain
\begin{equation*}
\{|\xi|^{\varepsilon} \widehat{u} (\xi)\}_{\xi \in \Z} \not\in \ell^2 (\Z),
\end{equation*}
and thus $\H^\infty _Q(\T) \not \subset \H^{\varepsilon} (\T)$.

\end{proof}

\begin{corollary}
If we set
\begin{equation*}
  Q(x,D_x)= \log^\rho(2+|D_x|), \ \ \rho > 0, x \in \T^n,
\end{equation*}
then, for every $\epsilon>0$, $\H^\infty _Q(\T^n) \not \subset \H^{\varepsilon} (\T^n)$.
\end{corollary}

The definition of global hypoellipticity used by J. Hounie says that the operator $\tilde{L}_\rho$, defined in \eqref{example-hounie-fail2}, is globally hypoelliptic on $\T\times\T$ if, given $u \in C^\infty(\T; \H^{-\infty}_Q(\T))$,
$$Lu \in   C^\infty(\T; \H^{\infty}_Q(\T)) \ \Rightarrow \ u \in C^\infty(\T; \H^{\infty}_Q(\T)).$$

It follows, from theorem 2.1 (\cite{HOU79} p.238), that the operator $\tilde{L}_\rho$ is globally hypoelliptic on $\T\times\T$, in the sense of the definition above, if, and only if, $b(t)$ does not change sign in $\T$, regardless of the value $\rho>0.$

Thus, in this case, J. Hounie does not say anything about the $C^\infty\!$--global hypoellipticity of this operator, while our theorem states that the $C^\infty\!$--global hypoellipticity of $\tilde{L}_\rho$ depends on $b$ and $\rho$.

\section{Reduction to normal form}

In this section we are going to show that, under suitable conditions, it is possible to replace the study of the global hypoelliptcity of the operator $L$ by an operator with constant coefficients. In particular, we will prove the following theorem:

\begin{theorem}\label{GH-conjugation-full}
Suppose that the condition \eqref{boundhyp} holds, that is,
\begin{equation}\label{boundhyp-2}
  \limsup_{j \to \infty} \dfrac{|\nu_{j}|}{\log (j)} = \kappa < + \infty.
\end{equation}
Then the following statements are equivalent:
\begin{quote}
\begin{enumerate}
  \item[{\it i.}] $L=D_t + a(t)p(x,D_x) + ib(t)q(x,D_x)$ is (GH);
  \item[{\it ii.}] $L_{a_0,b_0} = D_t + a_0p(x,D_x) + ib_0 q(x,D_x)$ is (GH).
\end{enumerate}
\end{quote}
\end{theorem}

The proof of this theorem consists in  constructing an automorphism $\Psi$  of $\D(\T\times M)$ 
such that
\begin{equation}\label{conj}
\Psi^{-1} \circ L \circ \Psi = L_{a_0,b_0}.
\end{equation}

First, we show how the condition \eqref{boundhyp-2} allows us to reduce the imaginary part of $L$ to the normal form in the diagonal case. Next, we use the classical reduction of the real part of $L$ to the normal form, to attain the full reduction, shown above, in the diagonal case. Finally, we show how to reduce to normal form in the case of multidimensional eigenspaces.

\subsection{Reduction to normal form in the diagonal form \label{reduct-section}}

Consider the map
\begin{equation}\label{psi-conjuga-b1}
u \in \D(\T \times M) \longmapsto \Psi_b u \doteq \sum_{j \in \N} e^{(B(t)-b_0t) \nu_j }u_j(t)\varphi_j(x),
\end{equation}
where $B(t) = \int_{0}^{t}b(s)ds$ and $b_0 = (2\pi)^{-1} \int_{0}^{2\pi}b(\tau) d\tau.$

If $\Psi_b$ is a linear operator on $\D(\T\times M)$, then the expression
\begin{equation*}
u \in \D(\T \times M) \longmapsto \Psi_b^{-1} u \doteq \sum_{j \in \N} e^{-(B(t)-b_0t) \nu_j }u_j(t) \varphi_j(x),
\end{equation*}
defines the inverse of $\Psi_b$, and thus $\Psi_b$ is an automorphism of $\D(\T\times M)$. Therefore, it is enough to prove that $\Psi_b u \in \D(\T\times M)$, for any
$u \in \D(\T\times M)$.

To prove this statement, consider the following sequence of smooth periodic functions
\begin{equation*}
  \psi_j(t) \doteq e^{(B(t)-b_0 t) \nu_j}u_j(t), \ t \in \T \mbox{ and } j \in \N.
\end{equation*}

We will show that $\{\psi_j(t)\}$ satisfies the condition \eqref{2-prop-smooth-2} of proposition \ref{1-prop-smooth-2} for some integer $N$,
i.e.,
\begin{equation*}
|\partial_t^k \psi_j(t)| \leqslant C\, j^{N}, \mbox{ as } j \to \infty,
\end{equation*}

Observe that the derivatives of $\psi_j$ depend on the powers of $\nu_j$ and we can control this growth with the assistance of the following result, which will be useful in other proofs that will appear in this work.

\begin{proposition}\label{propseque}
$|\mu_j| = \O(j^{1/n})$ \ and  \ $|\nu_j| = \O(j^{1/n})$, as  $j \to \infty.$
\end{proposition}

\begin{proof}
Since $p(x,D_x)$ and $q(x,D_x)$ are continuous linear operators from the space $\H^1(M)$ to $\H^0(M)=L^2(M)$ then, by \eqref{Sobolev1matr}, we have
\begin{equation*}
\sum_{j\in \N} |u_j|^2 j^{2/n} <+\infty \Leftrightarrow u \in H^1(M) \Rightarrow p(x,D_x)u, q(x,D_x)u \in L^2(M).
\end{equation*}

Now, from \eqref{2-prop-smooth-2} and \eqref{Sobolev1matr}, we obtain
\begin{align*}
  \Big\|p(x,D_x)u\Big\|_{L^2(M)}^2 & = \Big\|\sum_{j\in\N} u_j \mu_j\varphi_j(x) \Big\|_{L^2(M)}^2  \\[2mm]
                                   & = \ \sum_{j\in\N} \Big(|u_j|^2 \mu_j^2 \Big)\\[2mm]
                                   & =  \sum_{j\in\N} \Big( \dfrac{\mu_j^2} {j^{{2}/{n}}} ( |u_j|^2 j^{{2}/{n}}) \Big).
\end{align*}

It follows from lemma \ref{lema-seq}, that the sequence $\big\{|\mu_j|j^{-1/n}\big\}$ is bounded, and therefore
\begin{equation*}
|\mu_j| = \O(j^{1/n}) \ \mbox{ as } \ j \to \infty.
\end{equation*}

Analogously, $|\nu_j| = \O(j^{1/n})$ as $j \to \infty$.

\end{proof}

\begin{lemma}\label{lema-seq}
Let $\{\omega_j\}_{j \in \N} $ be a sequence of complex numbers with the following property: for all sequence of complex numbers $\{u_j\}$,
\begin{equation*}
  \sum_{j \in \N} |u_j|^2 j^{2/n}< \infty \ \Longrightarrow \ \sum_{j \in \N} |\omega_j|^2|u_j|^2 j^{2/n}< \infty.
\end{equation*}
Then $\{\omega_j\}$ is bounded.
\end{lemma}

\begin{proof}
If $\{\omega_j\}$ was unbounded, we could construct a subsequence $\{\omega_{j_k}\}_k$ such that
$$|\omega_{j_k}| > 2^{k/2}, \ \ k \in \N.$$

Setting
\begin{equation*}
u_j = \left\{
  \begin{array}{ll}
    2^{-k/2}j_k^{-1/n}, & \hbox{if $j=j_k$, for some $k \in \N$;} \\[2mm]
    0, & \hbox{otherwise.}
  \end{array}
\right.
\end{equation*}

We would have $\sum_{j} |u_j|^2 j^{2/n}< \infty$ and $\sum_{j} |\omega_j|^2|u_j|^2 j^{2/n} = \infty$.

\end{proof}

\begin{theorem}\label{norm.form.theo}
If \eqref{boundhyp-2} holds, then $\Psi_b$ defined in \eqref{psi-conjuga-b1} is an automorphism of $\D(\T\times M)$.
\end{theorem}

\begin{proof}
Following the ideas in the beginning of this subsection, to prove this result it suffices to show that the sequence of functions
\begin{equation*}
  \psi_j(t) \doteq e^{(B(t)-b_0 t) \nu_j}u_j(t), \ t \in \T,
\end{equation*}
satisfy the condition \eqref{2-prop-smooth-2} of proposition \ref{1-prop-smooth-2} for some integer $N$.

Since $b$ is periodic and smooth, and $u \in \D(\T\times M)$, this same proposition \ref{1-prop-smooth-2} guarantees the existence of an integer $N_0$ and of a constant $C>0$ such that
\begin{equation}\label{ineq-psi-1}
|\partial_t^k \psi_j(t)| \leqslant C j^{N_1} e^{(B(t)-b_0t) \nu_j}, \mbox{ as } j \to \infty,
\end{equation}
where $N_1 = N_0 + k/n$.

Now, we observe that, the hypothesis
$$
   \limsup_{j \to \infty} \dfrac{|\nu_{j}|}{\log (j)} = \kappa < + \infty,
$$
is equivalent to the following statement: for all $\varepsilon >0$, there is $j_0 \in \N$, such that
\begin{equation}\label{log-hyp-norm}
|\nu_j| \leqslant \log( j^{\kappa + \varepsilon}), \ \forall j \geqslant j_0.
\end{equation}

If we set
$$\rho = \max_{t \in [0,2\pi]}\big(B(t)-b_0t\big) \mbox{ \ and \ } \delta = \min_{t \in [0,2\pi]}\big(B(t)-b_0t\big),$$
then only one of the following three possibilities occur:
\begin{equation*}\label{cases}
\rho \leqslant \delta \leqslant 0, \ \ \  0 \leqslant \rho \leqslant \delta \ \ \mbox{ or } \ \ \rho \leqslant 0 \leqslant  \delta.
\end{equation*}

Moreover, as $\nu_j \rightarrow +\infty$ or $\nu_j \rightarrow -\infty$, we will have only one of the following inequalities, respectively:
\begin{align}
&\rho \nu_j \leqslant (B(t)-b_0t) \nu_j \leqslant \delta \nu_j, \ \textrm{ or} \label{B-b0t-infty-1} \\[2mm]
&\delta \nu_j \leqslant (B(t)-b_0t) \nu_j \leqslant \rho \nu_j. \label{B-b0t-infty-2}
\end{align}

First, let us analyze the case $\nu_j \rightarrow +\infty$. If $\rho \leqslant  \delta \leqslant 0$, estimate \eqref{B-b0t-infty-1} implies
$e^{\delta \nu_j} \leqslant 1$ for $j$ large enough; thus by \eqref{log-hyp-norm}  the estimate \eqref{ineq-psi-1} becomes
\begin{equation}\label{ineq-psi-2}
|\partial_t^k \psi_j(t)| \leqslant  C \ j^{N_1}, \ j \to \infty.
\end{equation}

Now, if $0 \leqslant \rho \leqslant \delta$, or $\rho \leqslant 0  \leqslant \delta$, it follows that
\begin{align}\label{ineq-psi-3}
|\partial_t^k \psi_j(t)| \leqslant & C \ j^{N_1} e^{\delta \nu_j}  \\[2mm]
                         \leqslant & C \ j^{N_1} e^{\delta \log(j^{\kappa + \varepsilon})}\nonumber \\[2mm]
                         \leqslant & C j^{N_1 + \delta ( \kappa + \varepsilon)}, \ j \to \infty.\nonumber
\end{align}

On one  hand, when $\nu_j \rightarrow -\infty$, if  $0 \leqslant \rho \leqslant \delta$ the inequality \eqref{B-b0t-infty-2} implies $e^{\rho \nu_j} \leqslant 1$, for $j$ large enough, in a way that we recapture \eqref{ineq-psi-2}. But, if $\rho \leqslant \delta \leqslant 0$, or $\rho \leqslant 0 \leqslant\delta$, we obtain
\begin{align}\label{ineq-psi-4}
|\partial_t^k \psi_j(t)| \leqslant & C \ j^{N_1} e^{\rho \nu_j}\\[2mm]
                         \leqslant & C \ j^{N_1} e^{-\rho\log(j^{\kappa + \varepsilon})}\nonumber \\[2mm]
                         \leqslant & C j^{N_1 - \rho ( \kappa + \varepsilon)}, \ j \to \infty.\nonumber
\end{align}

Thus, fixed $\varepsilon>0$ and setting
\begin{equation*}
N = \max \{N_1, \  N_1 - \rho ( \kappa + \varepsilon), \  N_1 + \delta ( \kappa + \varepsilon)\},
\end{equation*}
it follows that
\begin{equation*}
|\partial_t^k \psi_j(t)| \leqslant C j^{N}, \mbox{ as } \ j \to \infty,
\end{equation*}
and then $\Psi_b u \in \D(\T \times M)$.

\end{proof}

\begin{corollary}\label{corollary-normform}
If \eqref{boundhyp-2} holds, $\Psi_b$ is an automorphism of\, $C^{\infty}(\T\times M)$.
\end{corollary}

\begin{proof}
If $u \in C^{\infty}(\T\times M)$, then the expression  (\ref{ineq-psi-1}) becomes
\begin{equation*}
|\partial_t^k \psi_j(t)| \leqslant C \ j^{-\eta + k/n} \ e^{(B(t)-b_0t) \nu_j}, \mbox{ as } \ j \to \infty,
\end{equation*}
for any $\eta>0$.  Thus, setting
\begin{equation*}
N_3 = \max \{k/n, \  k/n - \rho ( \kappa + \varepsilon), \  k/n + \delta ( \kappa + \varepsilon)\},
\end{equation*}
we obtain
\begin{equation*}
|\partial_t^k \psi_j(t)| \leqslant C j^{-\eta  + N_3}, \ j \to \infty,
\end{equation*}
that implies $\Psi_b u \in C^{\infty}(\T \times M)$.

\end{proof}

\begin{proposition}\label{GH-conjugation}
Suppose that the condition \eqref{boundhyp-2} holds and consider the following  operator
\begin{equation*}
L_{b_0} \doteq  D_t + a(t)p(x,D_x) + ib_0 q(x,D_x).
\end{equation*}
Then we have
\begin{enumerate}
        \item[i.] $Lu=f \Leftrightarrow L_{b_0}v=g$, where $v = \Psi_b^{-1} u$ and $g = \Psi_b^{-1} f$;

        \item[ii.] $\Psi_b^{-1} \circ L \circ \Psi_b = L_{b_0}$;

  \item[iii.] $L$ is (GH) if, and only if, $L_{b_0}$ is (GH).
\end{enumerate}
\end{proposition}

\begin{proof} To prove $i.$, let $u \in \D(\T \times M)$ and $f \doteq Lu$, setting $v = \Psi_b^{-1} u$ and $g = \Psi_b^{-1} f$  we have
\begin{align*}
L_{b_0}v = & \sum_{j \in \N} \Big\{ \big[D_t (e^{-(B(t)-b_0t) \nu_j} u_j(t)) \\[2mm]
           & \hspace{9mm} + a(t)\mu_je^{-(B(t)-b_0t) \nu_j} u_j(t) + i b_0 \nu_j e^{-(B(t)-b_0t) \nu_j} u_j(t) \big] \varphi_j(x) \Big\} \\[2mm]
     = & \sum_{j \in \N} \left\{ [D_t u_j(t) + a(t)\mu_j u_j(t) + i b(t) \nu_j  u_j(t) ]  e^{-(B(t)-b_0t) \nu_j} \varphi_j(x) \right \} \\[2mm]
                 = & \sum_{j \in \N} f_j(t)  e^{- (B(t)-b_0t)\nu_j} \varphi_j(x) \ = \ \Psi_b^{-1} f \ = \ g.
\end{align*}

The proof of the other direction is analogous.

To prove $ii.$, using the same notation above, we have
$$
\Psi_b^{-1} \circ L \circ \Psi_b( v ) = \Psi_b^{-1}L(u) =  \Psi_b^{-1} f = g = L_{b_0} v.
$$

Finally, given $v \in \D(\T \times M)$ such that $g = L_{b_0} v \in
C^{\infty}(\T \times M)$, since $\Psi_b$ is an automorphism of
$C^{\infty}(\T \times M)$, we have that $f = \Psi_b g$ is a smooth function on $\T \times M$. From $ii.$ we have $L u = f$, where
$v = \Psi_b^{-1} u$. Supposing that $L$ is (GH), we have that $u$ is smooth, hence $v$ is smooth and $L_{b_0}$ is (GH). The converse assertion is proved in the same way.

\end{proof}

\subsection{Reduction of the real part}

The idea here is essentially the same one that we used for the imaginary part. Indeed, it is somewhat simpler because it does not require any additional hypothesis about the growth of the sequence $\{\mu_j\}$. Furthermore, this type of reduction was widely used by several authors, for example: A. P. Bergamasco \cite{BERG99}, A. P Bergamasco \emph{et al.} \cite{BKNZ15} and  W. Chen and M.Y. Chi \cite{CC00}; for this reason, and because the statements and proofs are very similar to the case already proved, we are just going to state the following result without proof.

\begin{proposition}\label{norm.form.theo-real}
Define on $\D(\T\times M)$ the following map
\begin{equation*}
u \mapsto \Psi_a u \doteq \sum_{j \in \N} e^{-i (A(t)-a_0t) \mu_j }u_j(t)\varphi_j(x).
\end{equation*}
Then
\begin{enumerate}
  \item[i.]  $\Psi_a$ is an automorphism of $\D(\T\times M);$
  \item[ii.] $\Psi_a$ is an automorphism of $C^{\infty}(\T\times M);$
  \item[iii.] $L$ is (GH) if, and only if, $L_{a_0}=D_t + a_0p(x,D_x) + ib(t) q(x,D_x)$ is (GH).
\end{enumerate}
\end{proposition}

\subsection{Normal form on multidimesional eigenspaces}

In this subsection we show how to recapture the reduction to normal form in the case of multidimensional eigenspaces. Just as we did before, the idea here is to obtain an automorphism $\Psi_{a,b}$ of the space $C^{\infty}(\T\times M)$, such that
\begin{equation}\label{conj1}
\Psi_{a,b}^{-1} \circ L \circ \Psi_{a,b} = L_{a_0,b_0}.
\end{equation}

Using the same notation as in section \ref{reduct-diag}, let $\{e^{k}_j(x)\}_{k = 1}^{d_j}$  be a basis of the space $E_{\sigma_j}$, and for each $u \in \D(\T\times M)$ write
\begin{equation*}
u = \sum_{j\in \N} \left\langle u^j (t) , e^j (x) \right\rangle_{\C^{d_j}}.
\end{equation*}

Let $P_j,Q_j \in \C^{d_j \times d_j}$ be the matrices of $p(x,D_x)$ and $q(x,D_x)$ on the space $E_{\sigma_j}$, with respect to that basis, and define the real sequences
\begin{align*}
\{\mu_j\} \doteq & \ \{ \mu_1^1, \ldots, \mu_1^{d_1}, \mu_2^{1}, \ldots, \mu_2^{d_2}, \ldots, \mu_j^{1}, \ldots, \mu_j^{d_j}, \ldots \}, \\[2mm]
\{\nu_j\} \doteq & \ \{ \nu_1^1, \ldots, \nu_1^{d_1}, \nu_2^{1}, \ldots, \nu_2^{d_2}, \ldots, \nu_j^{1}, \ldots, \nu_j^{d_j}, \ldots \},
\end{align*}
where $\{\mu_j^l\}$ and  $\{\nu_j^l\}$ are the eigenvalues of $P_j$ and $Q_j$, respectively.

With these notations, for each $u \in \D(\T\times M)$ set
\begin{equation}\label{psi-conjuga-b}
\Psi_{b} u \doteq \sum_{j\in \N} \left\langle   e^{(B(t)-b_0 t) \, Q_j} u^j (t) , e^j (x) \right\rangle_{\C^{d_j}}.
\end{equation}

\begin{proposition}
If $\{\nu_j\}$ satisfies \eqref{boundhyp-2}, then $\Psi_b$ is an automorphism of the spaces $\D(\T\times M)$ and $C^{\infty}(\T\times M)$.
\end{proposition}

\begin{proof}
Note that $Q_j \sim D_{Q_j} = \left ( \nu_{j}^1,   \ldots,  \nu_{j}^{d_j} \right)$, thus using the same notation $\{e^j_k (x)\}$ for the basis where $Q_j$ is diagonal, we obtain
\begin{align*}
\Psi_{b} u = & \sum_{j\in \N} \left\langle   e^{(B(t)-b_0 t) \, D_{Q_j} } u^j (t) , e^j (x) \right\rangle_{\C^{d_j}} \\[2mm]
                 = & \sum_{j\in \N}\sum_{k=1}^{d_j} e^{\nu_j^k (B(t)-b_0 t) }u_j^k(t) e^j_k (x)
                                                                 =  \sum_{j\in \N}\sum_{k=1}^{d_j} \psi_j^k(t) e^j_k (x).
\end{align*}

Thus, from the one-dimensional case, each $\psi_j^{\ell}(t)$ satisfies the conditions which guarantee that $\Psi_{b}$ is well defined. Moreover,
$\Psi_{b} u \in C^{\infty}(\T\times M)$, if $u \in C^{\infty}(\T\times M)$.

\end{proof}

\begin{corollary}
The map
\begin{equation}\label{psi-conjuga-a}
 \D(\T \times M ) \ni u \longmapsto \Psi_{a} u \doteq \sum_{j\in \N} \left\langle
                                                      e^{- i (A(t)-a_0 t) \, P_j} u^j (t) , e^j (x) \right\rangle_{\C^{d_j}}
\end{equation}
defines an automorphism of $\D(\T\times M)$  and $C^{\infty}(\T\times M)$.
\end{corollary}

\begin{proposition}\label{GH-conjugation-mult}
Let $L_{b_0}$ be the operator
\begin{equation*}
L_{b_0} =  D_t + a(t)p(x,D) + ib_0 q(x,D).
\end{equation*}

If conditions \eqref{boundhyp-2} hold and  $[p(x,D_x), q(x,D_x)] = 0$, then
\begin{enumerate}
        \item $Lu=f$ if, and only if, $L_{b_0}v=g$, where $v = \Psi_b^{-1} u$ and $g = \Psi_b^{-1} f$;

        \item $\Psi_b^{-1} \circ L \circ \Psi_b = L_{b_0}$;

  \item $L$ is (GH) if, and only if, $L_{b_0}$ é (GH).
\end{enumerate}
\end{proposition}

\begin{proof}
Let  $u, f \in \D(\T \times M)$, such that $v = \Psi_b^{-1} u$ and $g = \Psi_b^{-1} f$. To simplify, set
\begin{equation*}
{\cal{M}}_j (t) = e^{- (B(t)-b_0 t) \, Q_j }, \ j \in \N.
\end{equation*}

Then, we have
\begin{align}\label{equiv-1}
L_{b_0} v = & \ D_t v + a(t) p (x,D_x) v + ib_0 q(x,D_x) v \nonumber\\[3mm]
          = & \sum_{j\in \N} \left\langle D_t v^j (t) + a(t) P_j v^j (t) + ib_0 Q_j v^j (t), \ e^j (x) \right\rangle_{\C^{d_j}} \nonumber\\[2mm]
                                        = & \sum_{j\in \N} \left\langle D_t \left( {\cal{M}}_j (t) u^j (t) \right) + a(t)P_j{\cal{M}}_j (t) u^j (t)+ ib_0 Q_j{\cal{M}}_j(t)u^j(t),e^j(x) \right\rangle  \nonumber \\[2mm]
                                        = & \sum_{j\in \N} \left\langle {\cal{M}}_j (t) \left( D_t u^j (t) + a(t) P_j u^j (t) +  ib(t) Q_j u^j (t) \right),
                                             e^j (x) \right\rangle.
\end{align}

Since any matrix commutes with its exponential, we obtain
\begin{equation*}
Q_j  {\cal{M}}_j (t) = Q_j e^{- (B(t)-b_0 t) \, Q_j } = e^{- (B(t)-b_0 t) \, Q_j }  Q_j =  {\cal{M}}_j (t) Q_j.
\end{equation*}

On the other hand, from $[p(x,D_x),q(x,D_x)]=0$, we have $P_j Q_j = Q_j P_j$, and then
\begin{equation}\label{P-comuta-M}
P_j  {\cal{M}}_j (t) = P_j e^{- (B(t)-b_0 t) \, Q_j } = e^{- (B(t)-b_0 t) \, Q_j }  P_j =  {\cal{M}}_j (t) P_j,
\end{equation}
thus, from \eqref{equiv-1},
\begin{align}\label{equiv-2}
L_{b_0} v = &   \sum_{j\in \N} \left\langle{\cal{M}}_j (t) \left (  D_t v^j (t) + a(t) P_j v^j (t) + ib_0 Q_j v^j (t) \right ) \ , \ e^j (x) \right\rangle_{\C^{d_j}} \nonumber \\[1mm]
          = &   \sum_{j\in \N} \left\langle  e^{- (B(t)-b_0 t) \, Q_j }    f^j (t) \ , \ e^j (x) \right\rangle_{\C^{d_j}} \nonumber \\[2mm]
                                        = & \Psi^{-1}_b f = g,
\end{align}
which  implies $L_{b_0} v =g$. The other equivalence is identical, thus (1) is done. The statements  (2) and (3) are identical to the one-dimensional case.

\end{proof}

\begin{corollary}
If $\{\nu_j\}$ satisfies \eqref{boundhyp-2}  and $[p(x,D_x),q(x,D_x)] = 0$, then $\Psi_{a,b} = \Psi_{a} \circ \Psi_{b}$ defines an automorphism of $\D(\T\times M)$  and $C^{\infty}(\T\times M)$. Moreover, $\Psi_{a,b}^{-1} \circ L \circ \Psi_{a,b} = L_{a_0,b_0}$.
\end{corollary}

\begin{remark}
A crucial point in the last proof is to obtain \eqref{equiv-2} from \eqref{equiv-1}, which is possible only because the hypothesis $[p(x,D_x),q(x,D_x)] =0$ implies that \eqref{P-comuta-M} holds. Therefore is not possible to conjugate $L$ and $L_{b}$, as in the one-dimensional case, without the commutation hypothesis.

For the same reason, the reduction to normal form can not work for the real part. Indeed, let
\begin{equation*}
L_{a_0} =  D_t + a_0p(x,D) + ib(t) q(x,D)
\end{equation*}
and define
\begin{equation*}
{\cal{N}}_j (t) = e^{ i (A(t)-a_0 t) \, P_j }, \ j \in \N.
\end{equation*}

Thus, following the same calculations above, we need to obtain
\begin{equation*}
Q_j  {\cal{N}}_j (t) = Q_j e^{i (A(t)-a_0 t) \, P_j } = e^{i (A(t)-a_0 t) \, P_j }  Q_j =  {\cal{N}}_j (t) P_j.
\end{equation*}
\end{remark}

\section{Proof of theorem \ref{main-thm-dim-1}}

In this section we are going to state and prove three theorems (\ref{prop-1}, \ref{prop-2} and \ref{prop-3}) that, together with theorem \ref{GH-conjugation-full}, are equivalent to theorem \ref{main-thm-dim-1} about the (GH) of the diagonal case with the separation of variables.

We start by recalling that
$$
  L = D_t + a(t)p(x,D_x) + ib(t)q(x,D_x),
$$
where $a,b \in C^\infty(\mathbb{T}), \ p, q \in \Psi^1(M)$ are self-adjoint and commute with the normal elliptic operator $E$.

The set $\{\varphi_j\}$ is an orthonormal basis for $L^2(M)$, formed by eigenfunctions of $E$, and we are supposing that the corresponding eigenspaces $E_{\lambda_j}$ have dimension one.

\smallskip
The sequences of real numbers $\{\mu_j\}$ and $\{\nu_j\}$ satisfy
\begin{equation*}
p(x,D_x)\varphi_j = \mu_j \varphi_j \ \textrm{ and } \ q(x,D_x)\varphi_j = \nu_j \varphi_j, \ \ j \in \N.
\end{equation*}
and
$$
  \lim_{j \to \infty} |\nu_j| = \infty. \vspace{3mm}
$$

\begin{remark}\label{L-GH-iff-iL-GH}
First, note that the study of the (GH) of $L$ is equivalent to the study of the (GH) of the operator
$$
iL = \partial_t + ia(t)p(x,D_x) - b(t)q(x,D_x).
$$
In propositions \ref{prop-1} and \ref{prop-2}, we study the (GH) of $iL$. The reason for this choice is that the terms of the Fourier coefficients, with respect to $x$, are somewhat simpler, and the notation is closer to that used in the differential case (present in most studies published).
\end{remark}

Let $u \in \D(\T \times M)$ be a distribution such that $f \doteq iLu \in C^{\infty}(\T \times M)$. Taking the $x$-Fourier
\eqref{Fourier2matr}, we obtain the following sequence of ordinary differential equations:
\begin{equation}\label{Eq1-propA}
\partial_t u_j(t) + c_j(t) u_j(t) = f_j(t), \quad t \in \T, \ j \in \N,
\end{equation}
where $c_j(t) = -\nu_j b(t) +i\mu_j a(t)$. We denote $c_j^0=-\nu_j b_0+ia_0\mu_j$,  $\forall j \in \N$.

For each $j \in \N$, such that $c^0_j \notin i \Z$, the equation (\ref{Eq1-propA}) has a unique solution that can be written as
\begin{equation}\label{sol1-Eq2-propA}
u_j(t)=(1 - e^{-2\pi c^0_j} )^{-1} \int_{0}^{2\pi}{e^{\int_{t}^{t-s} c_j(\tau)d\tau}f_j(t-s)ds},
\end{equation}
or equivalently as,
\begin{equation}\label{sol1-Eq3-propA}
u_j(t) =  (e^{2\pi c^0_j}  - 1 )^{-1} \int_{0}^{2\pi}{e^{\int_{t}^{t+s} c_j(\tau)d\tau}f_j(t + s)ds},
\end{equation}

Note that we need to study de behavior of all derivatives of solutions \eqref{sol1-Eq2-propA}, \eqref{sol1-Eq3-propA} and, specially,
the derivatives of the exponential terms.

\begin{proposition}\label{prop-exp}
Consider the primitive $C_j(t) \doteq -\nu_j B(t) +i\mu_j A(t)$, where
\begin{equation*}
A(t) \doteq \int_{0}^{t}a(s)ds, \textrm{ \ and \ \ } B(t) = \int_{0}^{t}b(s)ds.
\end{equation*}

For any $k\in \N_0$, there is a constant $C = C(k,a,b)>0$, such that
\begin{equation}\label{prop 3-0}
\left | \partial_t^k e^{C_j(t)} \right| \leqslant C j^{k/n}  e^{-\nu_j B(t)}, \textrm{ as } j \to \infty.
\end{equation}
\end{proposition}
\begin{proof} For $k=0$ this is evident. If (\ref{prop 3-0}) is true for $\l \in \{0, 1, \ldots, k\}$, it follows from theorem \ref{propseque} that
\begin{align*}
\left | \partial_t^{k+1} e^{C_j(t)} \right|
             \leqslant & \sum_{\l=0}^{k} \binom{k}{\l} |\partial_t^{\l}(e^{C_j(t)}) \ \partial_t^{k-\l}( -\nu_j {b(t)} +i\mu_j a(t) ) |\\[2mm]
             \leqslant & \sum_{\l=0}^{k} \binom{k}{\l} C_{\l,a,b} \ j^{\l/n}  \ e^{-\nu_j B(t)} \\[2mm]
                       & \qquad \times \max \left\{\|\partial_t^{k-\l}b\|_\infty, \|\partial_t^{k-\l}b\|_\infty \right\}(|\nu_j| + |\mu_j|) \\[2mm]
                                           \leqslant & \sum_{\l=0}^{k} \binom{k}{\l} C'_{\l,a,b} \ j^{\l/n}  \ e^{-\nu_j B(t)} \ j^{\frac{1}{n}}
                                            \ \leqslant \ C_{k,a,b} \ j^{\frac{k+1}{n}} \ e^{-\nu_j B(t)}.
\end{align*}

\end{proof}

\begin{corollary}\label{coro-prop-exp}
For any $k\in \N_0$, there exist constants $C_1$ and $C_2$, depending only on $a$, $b$ and $k$, such that
\begin{align*}
\left| \partial_t^k \exp\left(\int_{t}^{t-s}c_j(\tau)d \tau\right) \right| \leqslant \ & C_1 \, j^{k/n} \exp\left(\nu_j\int_{t-s}^{t}b(\tau)d \tau\right), \mbox{ and} \\[3mm]
\left| \partial_t^k \exp\left(\int_{t}^{t+s}c_j(\tau)d \tau\right) \right| \leqslant \ & C_2 \, j^{k/n} \exp\left(-\nu_j\int_{t}^{t+s}b(\tau)d \tau\right),
\end{align*}
for all $s \in [0,2\pi],$ as $j \to \infty$.
\end{corollary}

\smallskip

Now, since $f$ is smooth, given any $\alpha \in \N_0$ and $\eta>0$, there is a positive constant $C$ and a natural number $j_0$, such that
\begin{equation}\label{f-smooth}
\sup_{t \in \T}|\partial^{\alpha}_t f_j(t)| \leqslant C \ j^{-\eta}, \ \ j \geqslant j_0.
\end{equation}

By corollary \ref{coro-prop-exp} and inequality (\ref{f-smooth}), for $k \in \N_0$, we have the following estimate to derivatives of
\eqref{sol1-Eq2-propA}
\begin{align}
|\partial^{k}_t u_j(t)| \leqslant & \ \Theta_j \int_{0}^{2\pi} \left| \partial^{k}_t \left(e^{\int_{t}^{t-s} c_j(\tau)d\tau}f_j(t-s)\right) ds \right| \nonumber \\[2mm]
                        \leqslant & \ \Theta_j \sum_{\l=0}^{k}\binom{k}{\l}
                                      \int_{0}^{2\pi}{ \Big| \partial^{\l}_t \big( e ^{\int_{t}^{t - s}c_j(\tau) d\tau} \big)\Big|
                                      \Big|\partial^{k-\l}_t f_j(t - s) \Big| ds} \nonumber  \\[2mm]
                        \leqslant &\ C \ \Theta_j \ j^{-\eta} \ \sum_{\l=0}^{k}\binom{k}{\l} \int_{0}^{2\pi}{ j^{l/n} \ e^{\nu_j\int_{t-s}^{t}b(\tau)d \tau} ds} \nonumber  \\[2mm]
  \leqslant & \ C \ \Theta_j \ j^{-\eta + k/n} \ \int_{0}^{2\pi} e^{\nu_j\int_{t-s}^{t}b(\tau)d \tau} ds. \label{estim1-derivatives}
\end{align}
where $\Theta_j = | 1 - e^{ - 2\pi c^0_j}|^{-1}.$

Analogously, for \eqref{sol1-Eq3-propA}, we have

\begin{align} \label{estim2-derivatives}
|\partial^{k}_t u_j(t)| \leqslant & \ \Theta_j \ e^{\nu_j 2 \pi b_0} \sum_{\l=0}^{k}\binom{k}{\l}
         \int_{0}^{2\pi}{ \Big| \partial^{\l}_t \big( e ^{\int_{t}^{t+s}c_j(\tau) d\tau} \big)\Big|
                                      \Big|\partial^{k-\l}_t f_j(t+s) \Big| ds}  \nonumber \\[2mm]
 \leqslant &\ C \ \Theta_j  \ e^{\nu_j 2 \pi b_0} \ j^{-\eta} \ \sum_{\l=0}^{k}\binom{k}{\l} \int_{0}^{2\pi}{ j^{l/n} \ e^{-\nu_j\int_{t}^{t+s}b(\tau)d \tau} ds} \nonumber \\[2mm]
 \leqslant & \ C \ \Theta_j \ e^{\nu_j 2 \pi b_0}  \ j^{-\eta + k/n} \ \int_{0}^{2\pi} e^{-\nu_j\int_{t}^{t+s}b(\tau)d \tau} ds.
\end{align}

From the next proposition we obtain the growth of sequences $\{\Theta_j\}$ and $\{\Theta_je^{\nu_j 2 \pi b_0}\}$.
\begin{proposition}\label{prop-theta}
If $b_0<0$, then
$$
\begin{array}{ll}
  \displaystyle \lim_{j \to +\infty} \Theta_j = 1,                    & \mbox{ if } \ \nu_j \to + \infty, \mbox{ and}\\ \\
  \displaystyle \lim_{j \to +\infty} \Theta_j e^{\nu_j 2\pi b_0} = 1, & \mbox{ if } \ \nu_j \to - \infty.
\end{array}
$$
\end{proposition}

\begin{proof}
Observe that $$\Theta_j=\left( e^{\nu_j 4 \pi b_0}-2 e^{\nu_j 2\pi b_0}\cos(2\pi a_0\mu_j) + 1\right)^{-1/2},$$ and that the two exponencial terms go to zero, as $j \rightarrow \infty$, provided that $b_0<0$ and $\nu_j \rightarrow + \infty$.

Analogously, $$\Theta_j e^{\nu_j 2\pi b_0} = \left(1 -2 e^{-\nu_j 2\pi b_0}\cos(2\pi a_0\mu_j) + e^{-4 \nu_j \pi b_0}\right)^{-1/2},$$ and here the exponential terms also go to zero as $j \rightarrow \infty$, if $b_0<0$ and $\nu_j \rightarrow - \infty$.

\end{proof}

\smallskip

The next step is to present and demonstrate the three theorems that we have announced. In order to do this, we are going to split the proofs in two subsections, namely: \textit{(GH) and Diophantine phenomena}, and \textit{Change of sign and super-logarithmic growth}.

\subsection{(GH) and Diophantine phenomena}

\begin{theorem}\label{prop-1}
If $b$ does not change sign and $b \not \equiv 0$, then  $L$ is (GH).
\end{theorem}

\begin{proof}
Note that
$$
  c^0_j \in i \Z \ \Leftrightarrow \ b_0 \nu_j = 0 \mbox{ and } a_0 \mu_j \in \Z.
$$

Since $b\not \equiv 0$, $b$ does not change sign and $|\nu_j| \to \infty$, thus $b_0 \neq 0$ and $\nu_j=0$ only for a finite number of indexes.
It follows that the set
$$
\{ j \in \N; c^0_j \in i \Z\} \mbox{ \ is finite.}
$$

Hence, to prove that the $x$-Fourier coefficients $u_j(t)$ satisfy (\ref{2-prop-smooth-2}), it is enough to study the behaviour of the solutions (\ref{sol1-Eq2-propA}) or
(\ref{sol1-Eq3-propA}), for $j$ sufficiently large.

\smallskip
We can assume, without loss of generality, that
\begin{equation}\label{b-leq-0-imply-b0<0}
 b(t)\leqslant 0, t \in [0,2\pi], \mbox{ \ wich implies \ }   b_0 < 0.
\end{equation}

Indeed, if $b(t)\geqslant 0$, by the change of variables $(t, x) \mapsto (-t, x )$, the operator $iL$ becomes
\begin{equation*}
i\widetilde{L} = -\partial_t + i\widetilde{a}(t)p(x,D_x) - \widetilde{b}(t) q(x,D_x),
\end{equation*}
where  $\widetilde{b}(t) = - b(-t) \leqslant 0$, and clearly, $iL$ is (GH) if, and only if, $i\widetilde{L}$ is (GH).

\smallskip
Now, to finish the proof of the proposition, let us show that the derivatives of $u_j$ satisfy the condition \eqref{2-prop-smooth-2} of proposition \ref{prop-smooth-1}, by separately analyzing their behaviour when ${\nu_j \to + \infty}$ and ${\nu_j \to - \infty}$.

\smallskip
First, let $s_0 \in [0,2\pi]$ be the point of maximum of $b$, that is,
$$
b(s_0) = \max_{ t \in [0,2\pi]}b(t) \leqslant 0.
$$

Then, for all $s \in [0,2\pi]$, we have
\begin{equation*}
\int_{t-s}^{t}b(\tau) d\tau \leqslant b(s_0)s \leqslant 0 \ \mbox{ and } \ \int_{t}^{t+s}b(\tau) d\tau \leqslant  b(s_0)s \leqslant 0.
\end{equation*}

When ${\nu_j \to + \infty}$, there is a natural $j_1$ such that $\nu_j > 0,$ for all $j \geqslant j_1,$ hence
$$
  \int_{0}^{2\pi} e^{\nu_j\int_{t-s}^{t}b(\tau)d \tau} ds \leqslant 2\pi, \ \ j \geqslant j_1.
$$

By proposition \ref{prop-theta}, the sequence $\{\Theta_j\}$ is bounded, therefore, from \eqref{estim1-derivatives} we have
\begin{equation}\label{estimate1-Dkuj}
  |\partial^{k}_t u_j(t)| \leqslant  C \ j^{-\eta + k/n}, \ \ j \geqslant j_1.
\end{equation}

When $\nu_j \to - \infty$, we use the equivalent expression \eqref{sol1-Eq3-propA}, which gives us the estimate \eqref{estim2-derivatives}. In this case,
there is a natural $j_2$ such that $\nu_j < 0,$ for all $j \geqslant j_2,$ hence
\begin{equation*}
\int_{0}^{2\pi} e^{-\nu_j\int_{t}^{t+s}b(\tau)d \tau} ds \leqslant 2\pi, \ \ j \geqslant j_2.
\end{equation*}

By proposition \ref{prop-theta}, the sequence $\{\Theta_j e^{\nu_j 2\pi b_0}\}$ is bounded and we have
\begin{equation}\label{estimate2-Dkuj}
|\partial^{k}_t u_j(t)| \leqslant \ C \ j^{-\eta + k/n}, \ \ j \geqslant j_2,
\end{equation}
for some $j_2 \in \N$.

\smallskip
Finally, from \eqref{f-smooth}, \eqref{estimate1-Dkuj} and \eqref{estimate2-Dkuj}, given any $k \in \N_0$, there is a constant $C>0$ and a positive integer $\eta$ satisfying $-\eta + k/n \leqslant -N$, such that
\begin{equation*}
|\partial^{k}_t u_j(t)| \leqslant  C \ j^{-N}, \ j \geqslant j_0,
\end{equation*}
where $j_0 = \max \{ j_1 , j_2 \}$.

\end{proof}


\begin{theorem}\label{prop-2}
If $b\equiv 0$, then the operator $L$ is (GH) if, and only if, the set $\Gamma_{a_0} \doteq \{ j\in \N; \ \mu_j a_0 \in \Z \}$ is finite and
$a_0$ is non-Liouville with respect to the sequence $\{\mu_j\}$.
\end{theorem}

\begin{proof} Let $u \in \D(\T \times M)$ be a solution of $iL u = f\in C^{\infty}(\T \times M)$. Following the steps in the introduction of this section, we are led to the following sequence of ordinary differential equations
\begin{equation}\label{odepropB}
\partial_tu_j(t) + i a(t) \mu_j u_j(t) = f_j(t), \ j \in \N.
\end{equation}

Since $\Gamma_{a_0}$ is finite, we have $c^0_j = i a_0 \mu_j \in i\Z$ only for a finite number of indexes, hence it is enough to study the solutions (\ref{sol1-Eq2-propA}) or (\ref{sol1-Eq3-propA}), when $j\to \infty$.

\smallskip
By hypothesis $b\equiv 0$, then $b_0=0$ and both expressions \eqref{estim1-derivatives} and \eqref{estim2-derivatives} become
\begin{equation} \label{propB-ineq1}
|\partial^{k}_t u_j(t)| \leqslant C \ j^{-\eta + k/n} \ \Theta_j, \ \forall \eta>0, \ \ j \to \infty.
\end{equation}

The next result studies the growth of sequence $\{\Theta_j\}$ (the proof is given below).
\begin{proposition}\label{prop-liou}
Let  $\{\beta_j \}_{j \in \N}$ be a sequence of real numbers. Then, for each $j \in \N$ there exist $\l(j) \in \Z$ such that
\begin{equation}\label{prop 2-0}
|1 - e^{2\pi i \beta_j}| \geqslant 4 \ | \beta_j + \l(j)|.
\end{equation}
\end{proposition}

\smallskip

Now, by \eqref{prop 2-0} we obtain, for any $j \in \N$, an integer $\l(j)$ such that
$$|1 - e^{-2\pi c^0_j}| \geqslant 4 \ |a_0 \mu_j + \l(j)|.$$

Thus, the hypothesis \eqref{nonLiouville} implies
\begin{eqnarray}\label{propB-ineq2}
\Theta_j = |1 - e^{-2\pi c^0_j}|^{-1} & \leqslant & C \ |a_0 \mu_j + \l(j)|^{-1} \\[3mm]
                                      & \leqslant & C \inf_{\l \in \Z} |a_0 \mu_j +\l|^{-1} \ \leqslant \  C j^{\delta}, \nonumber
\end{eqnarray}
for $j$ sufficiently large.

\smallskip
Thus, by (\ref{propB-ineq1}) and (\ref{propB-ineq2})  we obtain
\begin{equation*}
|\partial^{k}_t u_j(t)| \leqslant C \ j^{-\eta + k/n + \delta},  \ \forall \eta>0, \ j \to \infty,
\end{equation*}
which implies $u \in C^{\infty}(\T \times M)$.

\smallskip
To prove the sufficiency, let $L_{a_0}$ be the operator
\begin{equation*}
L_{a_0} = D_t + a_0 p(x,D_x).
\end{equation*}

Let us suppose by contradiction that $\Gamma_{a_0}$ is an infinite set and write
\begin{equation*}
\Gamma_{a_0} = \{ j_1 < j_2 < \ldots < j_k < \ldots \}.
\end{equation*}

Obviously we are supposing that $a_0 \neq 0,$ otherwise $L_{a_0}=D_t$ is trivially non (GH).

Consider the following sequence of functions in $C^{\infty}(\T)$
\begin{equation*}
u_j(t) \doteq \left\{
\begin{array}{l}
e^{-ia_0\mu_{j_k} t}, \textrm{ if } j = j_k, \\[3mm]
0, \textrm{ if } \ j\neq j_k.
\end{array} \right.
\end{equation*}

We have $|u_{j_k}(t)|\equiv 1$ for any $k \in \N$, and fixed $\l \in \N_0$ we obtain
\begin{equation*}
|\partial^{\l}_t u_{j_k}(t)| = |a_0 \mu_{j_k}|^{\l} \leqslant \ C \ j_{k}^{\l / n}, \ k \to \infty,
\end{equation*}
therefore $u_j(t)$ defines an element $u \in \D(\T \times M) \setminus C^{\infty}(\T \times M)$. On the other hand,
\begin{equation*}
 L_{a_0} \left (\sum_{j \in \N} u_j(t) \varphi_j(x)  \right) \
          = \sum_{k \in \N} [D_t(e^{-ia_0\mu_{j_k} t}) + a_0\mu_{j_k}e^{-ia_0\mu_{j_k}t}] \varphi_{j_k}(x) \ = \ 0,
\end{equation*}
hence $L_{a_0}$ is not (GH) and by reduction to normal form $L$ is not (GH).

\

Now, let us suppose by contradiction that $a_0$ is Liouville with respect to the sequence $\{\mu_j\}$. In this case, there is a subsequence
$\{\mu_{j_k}\}$ and a sequence  $\{\tau_k\} \subset \Z$ such that
\begin{equation}\label{non-Dioph-1}
|a_0 \mu_{j_k} - \tau_k | < {j_k}^{-k/2}, \ k \to \infty.
\end{equation}

Particularly, it follows from (\ref{non-Dioph-1}) that
\begin{equation}\label{non-Dioph-2}
|\tau_k | = \O({j_k}^{-k/2 + 1/n}), \ k \to \infty.
\end{equation}

Define sequences of functions $\{u_j(t)\}$ and $\{f_j(t)\}$ by
\begin{eqnarray*}
u_j(t) & = & \left\{
\begin{array}{l}
e^{-i\tau_k t}, \textrm{ if } j = j_k, \\[2mm]
0, \textrm{ otherwise. }
\end{array} \right.
\\
\
\\
f_j(t) & = & \left\{
\begin{array}{l}
(a_0 \mu_{j_k} - \tau_k)e^{-i\tau_k t}, \textrm{ if } j = j_k, \\[2mm]
0, \textrm{ otherwise. }
\end{array} \right.
\end{eqnarray*}

Note that $|u_{j_k}(t)| \equiv 1$, for any $t \in \T,$ and from (\ref{non-Dioph-2})
\begin{align*}
|\partial^{\l}_t u_{j_k}(t)| =  |\tau_k|^{\l} \leqslant C \ {j_k}^{-k/2+ 1/n}, \ k \to \infty,
\end{align*}
for any $\l \in \N$. Thus,  $\{u_j(t)\}$ defines an element $u \in \D(\T \times M) \setminus C^{\infty}(\T \times M)$.

But, from (\ref{non-Dioph-1}) and (\ref{non-Dioph-2}), for any  $\l \in \N_0$, we have:
\begin{align*}
|\partial^{\l}_t f_{j_k}(t)| \leqslant & |\tau_k|^{\l} |a_0 \mu_{j_k} - \tau_k| \\[2mm]
                             \leqslant & C \ j_k^{-k/2} \ {j_k}^{-k/2 + 1/n} \\[2mm]
                             \leqslant & C \ {j_k}^{-k + 1/n}, \ k \to \infty,
\end{align*}
then  $\{f_j(t)\}$ defines a function $f \in C^{\infty}(\T \times M)$,  such that $L_{a_0} u = f$; thus  $L_{a_0}$ is not (GH) and consequently $L$ is not (GH).

\end{proof}

\begin{corollary}
Admit that $\{\nu_j\}$ satisfies \eqref{boundhyp-2}. If $b_0\neq 0$, then $L$ is (GH). Otherwise, if $b_0=0$, $L$ is (GH) if, and only if,
\begin{enumerate}
\item[i.]  $\Gamma_{a_0}$ is a finite set, and

\item[ii.]  $a_0$ is non-Liouville with respect to the sequence $\{\mu_j\}$.
\end{enumerate}
\end{corollary}

\begin{proof}
From condition \eqref{boundhyp-2} we can apply the reduction to normal form which implies that $L$ is (GH) if, and only if, $L_{a_0,b_0}$ is (GH). Now, from propositions \ref{prop-1} and \ref{prop-2} we are going to obtain the proof.

\end{proof}

\subsection*{Proof of proposition \ref{prop-liou}}

First, for any $j \in \N$, there exists an integer $\l(j)$ such that $|\beta_j + \l(j)| \leqslant \frac{1}{2}$.
Next, observe that
$$|1-\cos(2 \pi t)| \geqslant \frac{2}{\pi}|t|, \mbox{ when } \frac{\pi}{2} \leqslant |t| \leqslant \pi,$$
and
$$|\sin(2\pi t)| \geqslant \frac{2}{\pi}|t|,, \mbox{ when } |t| \leqslant \frac{\pi}{2}.$$

Thus, when $\pi / 2 \leqslant |2\pi(\beta_j + \l(j))| \leqslant \pi$ \ we have
$$|1 - e^{2\pi i \beta_j}| \geqslant |1-\cos(2\pi(\beta_j+\ell(j)))| \geqslant 4|(\beta_j+\ell(j))|,$$
and when $|2\pi [ \beta_j + \l(j)]| \leqslant \pi / 2$ \ we have
$$|1 - e^{2\pi i \beta_j}| \geqslant |\sin(2\pi(\beta_j+\ell(j)))| \geqslant 4|(\beta_j+\ell(j))|.$$

\subsection{Change of sign and super-logarithmic growth}

\begin{theorem}\label{prop-3}
Suppose that the sequence $\{\nu_j\}$ has a subsequence $\{\nu_{j_k}\}$ such that
   \begin{equation}\label{onboundhyp-prop}
           \lim_{k \to \infty} \dfrac{|\nu_{j_k}|}{\log (j_k)} = + \infty.
   \end{equation}

\noindent If $b$ changes sign, then $L$ is not (GH).
\end{theorem}

Our strategy to this proof is to construct a singular solution to the
equation $Lu=f$. For this, we are going to present a sequence of smooth functions $\{u_j\}$ defined on $\T$, such that
$$
  u = \sum_{j\in\N} u_j \varphi_j \in \D(\T \times M) \setminus C^\infty(\T \times M),
$$
and $f \doteq Lu \in C^\infty(\T \times M)$.

This requires the use of the following lemma:
\begin{lemma}\label{lamma-chang.sig}
Let $b$ be a smooth real $2\pi$-periodic function on $\R$, such that $b \not \equiv 0$ on any interval. Then, the following properties are equivalent:
\begin{enumerate}
\item[i.] $b$ changes sign;

\item[ii.] there exists $t_0\in \R$ and $t^*, t_* \in ]t_0, t_0+2\pi[$ such that
\begin{align*}
 B_{t^*}(t) & \leqslant   0,  \ \forall t \in \   ]t_0,  t_0 + 2\pi], and      \\[2mm]
 B_{t_*}(t) & \geqslant   0,  \ \forall t \in \ ]t_0,t_0 + 2\pi[;
\end{align*}

\item[iii.] there exists $t_0 \in \R$, partitions
\begin{align*}
& t_0 < \alpha^* < \gamma^* < t^* <\delta^* <\beta^* < t_0 +2\pi,  \\[1mm]
& t_0 < \alpha_* < \gamma_* < t_* <\delta_* <\beta_* < t_0 +2\pi,
\end{align*}
and  positive constants $c^*, c_*$ such that the following estimates hold
\begin{align}
& \max_{t\in [\alpha^*, \gamma^*] \bigcup [\delta^*,\beta^*] } B_{t^*} (t)  < -c^*, and  \label{ch-sign-max1b}\\[1mm]
& \min_{t\in [\alpha_*, \gamma_*] \bigcup [\delta_*,\beta_*] } B_{t_*} (t)  >  c_*.   \label{ch-sign-min1b}
\end{align}
\end{enumerate}
\end{lemma}

\begin{proof}
We will only prove that $(i) \Rightarrow (ii) \Rightarrow (iii)$. The proof of $(iii) \Rightarrow (i)$ is trivial.

First, note that any primitive $ B_{\eta}(t) = \int_{\eta}^t b(s) ds $ satisfies
\begin{equation}\label{prim.1}
 B_{\eta}(t) =  B_{\zeta}(t) -  B_{\zeta}(\eta).
\end{equation}

Choose $t_0$ such that
\begin{equation*}
 b(t) <  0 , \ \ \textrm{ for $t $ near $t_0$},
\end{equation*}
and suppose $b_0>0$, which implies $B_{t_0}(t_0+2\pi) >0$.

We have
\begin{equation*} \label{prim.2}
 \frac{d}{dt} B_{t_0}(t) = b(t)  < 0,  \  \textrm{ for $t$ near to $t_0$ and near to $t_0 +2\pi$},
\end{equation*}
then $B_{t_0}$ decreases in a neighborhood of these two points, with $B_{t_0}(t_0) = 0$. Taking this and the periodicity of $b$ into account, we can find $\delta>0$ such that
\begin{equation*}\label{prim.3}
B_{t_0}(t)  \leqslant  B_{t_0}(t_0) =0, \ \textrm { for } \ t \in ]t_0, t_0+\delta[,
\end{equation*}
and
\begin{equation*}\label{prim.4}
B_{t_0}(t)  \geqslant  B_{t_0}(t_0 + 2\pi)>0,   \ \textrm { for } \ t \in ]t_0+ 2\pi - \delta, t_0+ 2\pi[.
\end{equation*}

Thus, there are points $t^*$ and $t_*$ such that
\begin{align*}\label{prim.5}
& B_{t_0}(t^*) \doteq \ \max\big\{B_{t_0}(t); t \in ]t_0, t_0+ 2\pi[ \,\big\}, \\[2mm]
& B_{t_0}(t_*) \doteq \ \min\big\{B_{t_0}(t); t \in ]t_0, t_0+ 2\pi[ \,\big\}.
\end{align*}

Then, by \eqref{prim.1}, we obtain
\begin{align*}
B_{t^*}(t) & =  B_{t_0}(t) -  B_{t_0}(t^*) \leqslant   0,  \ \forall t \in \ ]t_0, t_0+ 2\pi[,       \\[2mm]
B_{t_*}(t) & =  B_{t_0}(t) -  B_{t_0}(t_*) \geqslant   0,  \ \forall t \in \ ]t_0, t_0+ 2\pi[.
\end{align*}

\smallskip

Finally, because  $B_{t_0}$ is not constant on each of the intervals $]t_0, t^*[$, $]t_0, t_*[$, $]t^*, t_0+2\pi[$, $]t_*, t_0+2\pi[$, we obtain the strict inequalities in $(iii)$.

Now, if  $b_0<0$ and we choose $t_0$ such that $b(t)>0$, for $t$ near $t_0$, then the arguments of the proof are going to work, with obvious modifications.

\end{proof}


\subsection*{Proof of theorem \ref{prop-3}}

With the same notation of lemma \ref{lamma-chang.sig}, set the intervals
\begin{equation*}
I^* \doteq [\alpha^*, \gamma^*] \cup [\delta^*,\beta^*] \ \ \textrm{ and } \ \
I_* \doteq [\alpha_*, \gamma_*] \cup [\delta_*,\beta_*],
\end{equation*}
and choose $g^*,g_*, \psi^*,\psi_* \in C^{\infty}(\T)$ such that
\begin{align*}
&  \mbox{supp}(\psi^*) \subset [0,2\pi] \ \mbox{ and } \ \psi^*|_{[\alpha^*, \beta^*]}\equiv 1, \\[2mm]
&  \mbox{supp}(g^*) \subset [\alpha^*, \beta^*]  \ \mbox{ and } \ g^*|_{[\gamma^*,\delta^*]}\equiv 1,
\end{align*}
and
\begin{align*}
&  \mbox{supp}(\psi_*) \subset [0,2\pi] \ \mbox{ and } \ \psi_*|_{[\alpha_*, \beta_*]}\equiv 1, \\[2mm]
&  \mbox{supp}(g_*) \subset [\alpha_*, \beta_*]  \ \mbox{ and } \ g_*|_{[\gamma_*,\delta_*]}\equiv 1.
\end{align*}

Now, admit that $\nu_{j_k} \to +\infty$ and define a sequence $\{u_j\} \subset C^{\infty}(\T)$ by
\begin{equation*}
u_j(t) = \left\{
\begin{array}{l}
g^*(t) e^{\nu_{j_k} B_{t^*}(t)\psi^*(t) - i \mu_{j_k} A_{t^*}(t)\psi^*(t)}, \ \textrm{ if } \ j = j_k  \textrm{ for some } k \in \N; \\[2mm]
0, \ \textrm{ otherwise.}
\end{array} \right.
\end{equation*}

Note that, for any $t \in \mbox{supp}(g^*)$, we have
\begin{equation*}
g^*(t) e^{\nu_{j_k} B_{t^*}(t)\psi^*(t) - i \mu_{j_k} A_{t^*}(t)\psi^*(t)} =
g^*(t) e^{\nu_{j_k} B_{t^*}(t) - i \mu_{j_k} A_{t^*}(t)},
\end{equation*}
then for any $t \in \mbox{supp}(g^*)$ we have $e^{\nu_{j_k} B_{t^*}(t)\psi^*(t)}\leqslant 1$, for $k$ large enough, since we have $B_{t^*}(t) \leqslant 0$ on $I^*$ and $\nu_{j_k} \to +\infty$.

Therefore, for any $\beta \in \N$ and $t \in \mbox{supp}(g^*)$ we obtain
\begin{eqnarray*}
  \left| \partial_t^{\beta} u_{j_k}(t) \right| & \leqslant &
                            \sum_{\alpha \leqslant \beta}\binom{\beta}{\alpha}\left|\partial_t^{\beta-\alpha}\big(g^*(t)\big)\right| \
                      \left| \partial_t^{\alpha}\left(e^{\nu_{j_k} B_{t^*} (t) - i \mu_{j_k} A_{t^*}(t)}\right) \right| \\[2mm]
      & \leqslant & C_{a,b,g,\beta}\big(|\mu_{j_k}| + |\nu_{j_k}|\big)^\beta e^{\nu_{j_k} B_{t^*} (t)} \\[3mm]
      & \leqslant & C j^{\beta/n}_k, \mbox{ as }  k \to \infty.
\end{eqnarray*}

Since $|u_{j_k}(t^*)| = 1$, for any $k$, we have
\begin{equation}\label{u-singular-superlog}
  u \doteq \sum_{j \in \N} u_j(t)\varphi_j(x) \in \D(\T \times M) \setminus C^{\infty}(\T \times M).
\end{equation}

Next, we show that $f \doteq Lu \in C^{\infty}(\T \times M).$ Here $f(t,x) = \sum_{j} f_j(t)\varphi_j(x)$, where
\begin{equation*}
f_j(t) = \left\{
\begin{array}{l}
-i {g^{*}}'(t) e^{\nu_{j_k} B_{t^*}(t)\psi^*(t) - i \mu_{j_k} A_{t^*}(t)\psi^*(t)}, \ \textrm{ if } \ j = j_k, \mbox{ for some } k \in \N; \\[2mm]
0, \ \textrm{ otherwise}.
\end{array} \right.
\end{equation*}

Note that $ \supp(f_{j_k}) \subset I^*$, for any $k \in \N$ and
\begin{equation}\label{partial-beta-fjk}
|\partial^{\beta}_t f_{j_k}(t)| \leqslant C \ j_{k}^{\beta / n} \ e^{\nu_{j_k}B_{t^*}(t)}, \ \ k \to \infty.
\end{equation}

We observe that, at this point, we can not eliminate the exponential above, using the expression $e^{\nu_{j_k} B_{t^*}(t)\psi^*(t)}\leqslant 1$, because this would only ensure that $\{f_j\}$ is slow-growing and therefore is a periodic distribution. In order to obtain the rapid decreasing of coefficients $f_j$, we need to analyze the consequences of super-logarithmic growth, as in condition \eqref{2-prop-smooth-2}.

\smallskip

By the estimate \eqref{ch-sign-max1b} we obtain
\begin{equation*}
\nu_{j_k} B(t) \psi^*(t) \leqslant - \nu_{j_k} c^*, \ \forall t \in \supp(f_{j_k}).
\end{equation*}

Since \eqref{onboundhyp-prop} is equivalent to
\begin{equation}\label{onboundhyp_equiv-1}
(\forall \eta>0) (\exists \ k_0 \in \N)(\forall k \geqslant k_0) \ \log (j_k^{\eta}) < \nu_{j_k},
\end{equation}
it follows from \eqref{partial-beta-fjk} that
\begin{align*}
|\partial_t^{\beta} f_{j_k}(t)| & \leqslant C {j_k}^{\beta/n} e^{\nu_{j_k} B_{t^*}(t)} \\[2mm]
                                & \leqslant C {j_k}^{\beta/n} e^{ - \nu_{j_k} c^* }  \\[2mm]
                                & \leqslant C {j_k}^{\beta/n} e^{- c^* \log ({j_k}^{\eta})} \\[2mm]
                                & \leqslant C {j_k}^{-\eta c^* + \beta/n}, \mbox {if } k>k_0,
\end{align*}
for any $t \in \supp(f_{j_k})$.

Since $\eta$ can be chosen arbitrarily large, by taking a large enough $k$, thus $\{f_j\}$ satisfies \eqref{2-prop-smooth-2} and $f \in C^{\infty}(\T \times M)$ and therefore $L$ is not (GH).

This concludes the proof in the case where $\nu_{j_k} \to +\infty$. Now, observing the definition of $u$ above, it is not difficult to see that it is possible to substitute the condition $\nu_{j_k} \to +\infty$ by the weaker assertion that $\{\nu_{j_k}\}$ has a subsequence that diverges to $+\infty$.

On other hand, if $\nu_{j_k} \to -\infty$, we use the primitive $B_{t_*}$ and set

\begin{equation*}
u_j(t) = \left\{
\begin{array}{l}
g_*(t) e^{\nu_{j_k} B_{t_*}(t)\psi_*(t) - i \mu_{j_k} A_{t_*}(t)\psi_*(t)}, \ \textrm{ if } \ j = j_k  \textrm{ for some } k \in \N; \\[2mm]
0, \ \textrm{ otherwise.}
\end{array} \right.
\end{equation*}

In this case we obtain, by  estimate \eqref{ch-sign-min1b},
\begin{equation*}
e^{\nu_{j_k} B_{t_*}(t)}\psi_*(t) \leqslant e^{\nu_{j_k} c_*} \leqslant j_k^{-\eta c_*},
\end{equation*}
for any $t \in \mbox{supp}(g_*)$.

\section{The hypothesis of the unboundedness of $\{\nu_j\}$ \label{remark-mu}}

The purpose of this section is to show how to replace the hypothesis
\begin{equation}\label{A2-nu-j-1}
\lim_{j \to \infty} |\nu_j|  =  \infty
\end{equation}
by a weaker condition.

First, we emphasize that condition \eqref{A2-nu-j-1} was only used three times on the proof of our results, namely:
\begin{enumerate}

        \item [(a)] in the proof of theorem \ref{norm.form.theo} to obtain the inequalities
\begin{align*}
\rho \nu_j \leqslant (B(t)-b_0t) \nu_j \leqslant \delta \nu_j \ \ \textrm{ and } \ \
\delta \nu_j \leqslant (B(t)-b_0t) \nu_j \leqslant \rho \nu_j;
\end{align*}

        \item [(b)] in the proof of theorem \ref{prop-1} to obtain the estimates
\begin{enumerate}
  \item[{\it i.}] $\displaystyle \int_{0}^{2\pi} \ e^{\nu_j\int_{t-s}^{t}b(\tau)d \tau} ds \leqslant 2\pi,$ \ when $\nu_j \to + \infty$;
  \item[{\it ii.}] $\displaystyle \int_{0}^{2\pi} e^{-\nu_j\int_{t}^{t+s}b(\tau)d \tau} ds \leqslant 2\pi,$ when $\nu_j \to - \infty$; and  \vspace{2mm}
  \item[{\it iii.}] $b_0\nu_j = 0$ only for a finite number of indexes.
\end{enumerate}
  \item [(c)] in the proof of theorem \ref{prop-1} to guarantee that the sequences $\{\Theta_j\}$  and $\{\Theta_j \, e^{\nu_j 2\pi b_0} \}$ are bounded,  see \eqref{estimate1-Dkuj} and \eqref{estimate2-Dkuj} page \pageref{estimate1-Dkuj}.
\end{enumerate}

Keeping these points and their proofs in mind, our goal is to weaken condition \eqref{A2-nu-j-1} in order to preserve these estimates.  A first attempt in this direction is the following:
\begin{equation}\label{cond-bc}
\textit{suppose there is $C>0$ and  $j_0 \in \N$, such that $|\nu_j| \geqslant C, \ \forall j \geqslant j_0$.}
\end{equation}

If \eqref{cond-bc} holds, we can easily recapture the inequalities highlighted in items (a) and (b) above. However, to recover the inequalities in item (c), we have to analyze these expressions more carefully. Observe that the main point of (c) is to ensure that the sequences $\{\Theta_j\}$ and $\{\Theta_j e^{\nu_j 2\pi b_0}\}$ have a controlled growth when $j \rightarrow +\infty$, that is, the sequence
\begin{equation}\label{om-1}
\omega_j =  e^{ \nu_j 2 \pi b_0} \left( e^{ \nu_j 2 \pi b_0} -2 \cos(2\pi a_0\mu_j) \right)+ 1
\end{equation}
does not converge rapidly to zero.

Let us investigate what happens when this sequence convergences to zero. For this, admit that there is a subsequence $\omega_{j_k} \to 0$, for
$k \to \infty$. By formula \eqref{om-1} we have
\begin{equation*}
e^{ \nu_{j_k} 2 \pi b_0} < 2 \cos(2\pi a_0\mu_{j_k}) \leqslant 2, \ k \to \infty,
\end{equation*}
and then
\begin{equation*}
\nu_{j_k} \pi b_0  < \log (2), \ j \to \infty.
\end{equation*}

Thus, we can set $\kappa = \limsup_{k \in \N} \nu_{j_k}$ and $\{\nu_{j_{\ell}}\}_{\ell}$, such that
\begin{equation*}
\lim_{\ell \to \infty} \nu_{j_{\ell}} = \kappa  \ \ \textrm{ and } \ \ \lim_{\ell \to \infty} e^{ \nu_{j_{\ell}} 2 \pi b_0} = \alpha < 2.
\end{equation*}

Then, we obtain
\begin{align*}
     0 =& \lim_{\ell \to \infty} \omega_{j_{\ell}} \\[3mm]
       =&  \lim_{\ell \to \infty} e^{ \nu_{j_{\ell}} 2 \pi b_0} \left( e^{ \nu_{j_{\ell}} 2 \pi b_0} -2 \cos(2\pi a_0\mu_{j_{\ell}}) \right) +1 \\[2mm]
       =& \alpha  \left( \alpha -2 \lim_{\ell \to \infty} \cos(2\pi a_0\mu_{j_{\ell}}) \right) +1,
\end{align*}
and
\begin{equation}\label{mu-1}
\lim_{\ell \to \infty} \cos(2\pi a_0\mu_{j_{\ell}}) =  \dfrac{1 + \alpha^2}{2 \alpha}.
\end{equation}

But, from \eqref{mu-1} we have $1 + \alpha^2 \leqslant 2 \alpha$, implying $\alpha =1$ and then $\kappa = 0$.

Further, a necessary condition to $\{\omega_j\}$ approach to zero is that $\{\nu_j\}$ has a subsequence $\{\nu_{j_{\ell}}\}_{\ell}$ converging to zero. Moreover, $a_0 \mu_{j_{\ell}} \to z \in \Z$, when $\ell \to \infty$.

From this discussion, we have:
\begin{proposition}\label{hy-nu-fraca}
The hypothesis \eqref{A2-nu-j-1} can be replaced by the condition
\begin{equation}\label{hy-nu-fraca-1}
\textit{zero is not an accumulation point of the sequence $\{\nu_j\}$.}
\end{equation}
\end{proposition}

\smallskip

\begin{example}
Let $\tau \in \N$,  $c \in \Z_+$ and consider the sequence
\begin{equation*}
\mu_j = \dfrac{(c + j)^{\tau}}{j^{\tau}}, \ j \in \N,
\end{equation*}
for which there are $j_0 \in \N$ and $C'>0$ such that
\begin{equation}\label{mu-limitada}
0 < C' \leqslant \mu_j , \ \ \forall j\geqslant j_0.
\end{equation}

Consider the operator
\begin{equation*}
q(x,D_x) \cdot u = \sum_{j \in \N} u_j \mu_j \varphi_j(x),
\end{equation*}
an irrational number $\alpha \in \R$ and
\begin{equation*}
{\cal{P}} = D_t + \alpha q(x,D_x), \ \ (t,x) \in \T^2 = \T_t \times \T_x.
\end{equation*}

If $\alpha$ is a non-Liouville number, the operator ${\cal{P}}$ is (GH). Indeed, there exist $\delta>0$ such that
\begin{equation}\label{alpha-n-L}
\left | \alpha + \dfrac{p_j}{q_j} \right | \geqslant \dfrac{1}{|q_j|^{\delta}},
\end{equation}
then for each $\ell \in \Z$ we obtain
\begin{align*}
|\alpha \mu_j + \ell | & = \mu_j \Big | \alpha + \dfrac{\ell j^{\tau} }{(c + j)^{\tau}} \Big | \\[2mm]
                       &  \geqslant  \dfrac{C'}{(c + j)^{\tau \delta}}  \\[3mm]
                       & \geqslant {C}{j^{\, -\tau \delta}},
\end{align*}
which implies $\alpha$ non-Liouville with respect to $\{\mu_j\}$. Now, $\Gamma_{\alpha} = \emptyset$, since $\alpha \in \R \setminus \Q$, thus ${\cal{P}}$ is (GH) by theorem \ref{main-thm-dim-1}, item \textit{iii}, and proposition \ref{hy-nu-fraca}.
\end{example}

\smallskip

A natural question is: if the operator
\begin{equation*}
L = D_t + a(t)p(x,D_x) + i b(t)q(x,D_x)
\end{equation*}
does not satisfy the condition \eqref{hy-nu-fraca-1}, what are the  consequences in the study of global hypoellipticity?

Evidently, there are no novelties (in theorem \ref{main-thm-dim-1}) if $b\equiv 0$, or if $b$ changes sign and $\{\nu_j\}$ has
super-logarithmic growth, since in the first case, $\{\Theta_j\}$
depends only on $\{a_0\mu_j\}$, and in the second case, we can construct a singular solution, as shown in theorem \ref{prop-3}.

Thus, let us investigate the case $b \not \equiv 0$.

\begin{theorem}\label{theorem-nu-zero-1}
Admit that: the sequence $\{\nu_{j}\}$ has at most logarithmic growth; $\{ j \in \N; \nu_{j} = 0\}$ is finite; zero is an accumulation point of $\{\nu_{j}\}$. Then, the following statements are true:

\begin{enumerate}
        \item [i.] when $\displaystyle \liminf\omega_j \neq 0$ we have
            \begin{enumerate}
                     \item $b_0 \neq 0$ implies that $L$ is (GH);\vspace{2mm}
                         \item $b_0 = 0$ implies that $L$ is (GH) if, and only if, $\Gamma_{a_0}$ is finite and $a_0$ is non-Liouville with respect to the sequence $\{\mu_j\}$;
            \end{enumerate}

        \item[ii.] when $\displaystyle \liminf\omega_j = 0$, the operator $L$ is (GH) if, and only if,  $\Gamma_{a_0}$ is finite and $a_0$ is non-Liouville with
               respect to the sequence $\{\mu_j\}$.
\end{enumerate}
\end{theorem}

\smallskip

\begin{remark}
We emphasize that when $\displaystyle \liminf\omega_j = 0$, the global hypoellipticity depends only of the real part of $L$. Moreover, $b_0 \neq 0$ is not a sufficient condition for global hypoelipticity, even in the case of constant coefficients.
\end{remark}

\begin{proof}
Since $\{\nu_{j}\}$ has a logarithmic growth we can apply  the reduction to normal form; thus $L$ is (GH) if, and only if, the operator $L_{a_0,b_0}$ is (GH).

When $\displaystyle \liminf\omega_j \neq 0$, the sequence $\{\Theta_j\}$ is bounded, thus item \textit{i.} follows from theorem \ref{main-thm-dim-1}.

Now, admit that $\displaystyle \liminf\omega_j = 0$ and let $u \in \D(\T \times M)$ be a solution of $(iL_{a_0,b_0})u = f \in C^{\infty}(\T \times M)$. Then we obtain the sequence of differential equations
\begin{equation}\label{sol-theorem-nu-zero}
\partial_t u_j(t) + (-b_0\nu_j + i a_0\mu_j) u_j(t) = f_j(t), \ t \in \T, \ j\in \N.
\end{equation}

If $b_0 = 0$ this result is a consequence of theorem
\ref{prop-2}. On the other hand, if  $b_0\neq 0$, we can assume $b_0<0$, then the unique solutions of \eqref{sol-theorem-nu-zero} are given by \eqref{sol1-Eq2-propA} or \eqref{sol1-Eq3-propA}, since $\nu_{j} = 0$ at most for a finite number of indexes $j$.

From inequalities \eqref{estim1-derivatives} and \eqref{estim2-derivatives} follows that
\begin{align}
|\partial_t^m u_j(t)| & \leqslant C j^{-\eta + m/n} \Theta_j, \ \textrm{ for } \ \nu_j>0 \ \textrm{ and } \ \label{1-a}\\[2mm]
|\partial_t^m u_j(t)| & \leqslant C j^{-\eta + m/n} \Theta_j \, e^{2 \pi b_0 \nu_j}, \ \textrm{ for } \ \nu_j<0. \label{1-b}
\end{align}

Since $\displaystyle \liminf\omega_j = 0$, we obtain a subsequence $\{\omega_{j_k}\}$ converging to zero and, by the discussion that followed after equation \eqref{mu-1}, we have
$a_0 \mu_{j_k} \to \gamma \in \Z$ and
\begin{align*}
 0 = \lim_{k \to \infty} \omega_{j_{k}}  = & 2 \lim_{k \to \infty} \left( 1 -  \cos(2\pi a_0\mu_{j_{k}}) \right).
\end{align*}

When $j \neq j_k$ we can control the growth of $\{\Theta_j\}$ and $\{\Theta_j \, e^{2 \pi b_0 \nu_j}\}$ by using the same ideas shown in theorem \ref{prop-1}.

For $j=j_k$, we have $e^{2 \pi b_0 \nu_{j_k}} \to 1$, when $k \to \infty$, and it is sufficient to study the behaviour of $\{\Theta_j\}$.

Now, consider the inequality
\begin{equation}\label{ine-cos}
|1- \cos(y)| \geqslant  |y - 2\pi \ell|^3, \  \textrm{ if } \ |y - 2\pi \ell| \leqslant 1/2, \ \forall \ell \in \Z.
\end{equation}

Thus, if $a_0$ is non-Liouville with respect to the sequence $\{\mu_j\}_{j \in \N}$, we obtain by \eqref{ine-cos} that
\begin{align*}
\lim_{k \to \infty} \omega_{j_{k}}  \geqslant \, & 2 \lim_{k \to \infty} |2\pi a_0\mu_{j_{k}} - 2\pi \gamma |^3 \\[2mm]
                                    \geqslant \, & 16\pi^3 \lim_{k \to \infty} \left ( \inf_{\ell \in \Z}|a_0\mu_{j_{k}} + \ell| \right)^3\\[2mm]
                                    \geqslant \, & 16\pi^3 C' j_k^{\, 3 \delta}, \ k \to \infty.
\end{align*}

From this, there exists $k_0 \in \N$ such that $\Theta_{j_k} \leqslant \, C j_k^{-3 \delta/2}$, $\forall k\geqslant k_0$, which implies $L_{a_0,b_0}$ is (GH) and that the conditions, $\,\Gamma_{a_0}$ is finite and $a_0$ is non-Liouville with respect to the sequence $\{\mu_j\}_{j \in \N}$, are sufficient conditions for the (GH).

The necessity uses the same ideas shown in theorem \ref{prop-2}, where we have constructed singular solutions.

\end{proof}

\smallskip

\begin{example} \label{pro1-ex1}
Let $\tau \in \N$,  $c \in \Z_+$, and consider the sequence
\begin{equation*}
\mu_j = \dfrac{(c + j)^{\tau}}{j^{\tau + 1}}, \ j \in \N.
\end{equation*}

For each $u = \sum_{j \in \N} u_j  \varphi_j(x) \in \D(\T_x)$ set
\begin{equation*}
q(x,D_x) \cdot u = \sum_{j \in \N} u_j \mu_j \varphi_j(x).
\end{equation*}

Let $\alpha$ and $\beta$ be real numbers, with $\beta \neq 0$, and consider the operator
\begin{equation}\label{ex-op}
{\cal{P}} = D_t + ( \alpha  + i \beta) q(x,D_x), \ (t,x) \in \T^2 = \T_t\times \T_x.
\end{equation}

Thus, if $\alpha$ is an irrational non-Liouville number then the operator ${\cal{P}}$ is (GH).

Indeed, by \eqref{alpha-n-L} we obtain
\begin{equation*}
|\alpha \mu_j + \ell | \geqslant C j^{\tau(\delta -1) - \delta},
\end{equation*}
for each $\ell \in \Z$.
\end{example}

\smallskip

\begin{remark}
If $\alpha$ is an irrational Liouville number, the example above exhibits a class of non-(GH) operators on $\T^2$ of the type
\begin{equation*}
L = D_t + ( \alpha  + i \beta) q(x,D_x), \mbox{ \ with \ }  b\neq 0,
\end{equation*}
which is a phenomenon that does not occur in the differential case, see the works of S. Greenfield and N. Wallach \cite{GW1} and J. Hounie \cite{HOU79}.

Moreover, we point out that the sequence $\{\mu_j\}$ has growth at most logarithmic; thus by reduction to normal form, the conclusion above holds even for operators with variable coefficients, that is, there exist operators, with an imaginary part not identical to zero, that do not change sign and are non-(GH) on $\mathbb{T}^2$.
\end{remark}

\smallskip

The hypothesis of logarithmic growth, added to theorem \ref{theorem-nu-zero-1}, implies that it is enough to consider constant coefficients operators, therefore the next step is to study operators that do not satisfy this condition. We are going to start by considering the operator
\begin{equation*}
L = D_t + a(t)p(x,D_x) + i b(t)q(x,D_x).
\end{equation*}

If $b$ changes sign and $\{\nu_j\}$ has super-logarithmic growth, then $L$ is not (GH), independently of value of $\liminf\omega_j$. Indeed, in this case, one can construct a singular solution using the ideas shown on the proof of theorem \ref{prop-3}.

The next result includes the remaining cases.

\begin{theorem}\label{theorem-nu-zero-2}
Admit that the imaginary part $b$ does not change sign and is not identical to zero. Also, assume that $\{\nu_j\}$ has super-logarithmic growth. If $\{j\in\N; \nu_j = 0\}$ is finite, the following statements are true:

\begin{enumerate}
        \item [i.] when $\displaystyle \liminf \omega_j  \neq 0$, the operator $L$ is (GH);

        \item[ii.] when $\displaystyle \liminf\omega_j = 0$, the operator $L$ is (GH) if, and only if,  $\Gamma_{a_0}$ is finite and $a_0$ is non-Liouville with respect to the sequence $\{\mu_j\}$.
\end{enumerate}
\end{theorem}

\begin{proof}

When $\displaystyle \liminf \omega_j \neq 0$ the sequence $\{\Theta_j\}$ is bounded, thus item \textit{i.}\!\! is a consequence of theorem
\ref{main-thm-dim-1}.

Now, admit that $\displaystyle \liminf\omega_j = 0$ and let $u \in \D(\T \times M)$ be a solution of $(iL)u = f \in C^{\infty}(\T \times M)$. Then we obtain the sequence of differential equations
\begin{equation}\label{sol-theorem-nu-zero-2}
\partial_t u_j(t) + (-b(t)\nu_j + i a(t)\mu_j) u_j(t) = f_j(t), \ t \in \T, \ j\in \N.
\end{equation}

We can assume $b_0<0$. Since $\{j\in\N; \nu_j = 0\}$ is finite, then the unique solutions of \eqref{sol-theorem-nu-zero-2} are given by \eqref{sol1-Eq2-propA}, or \eqref{sol1-Eq3-propA}.

Thus, by inequalities \eqref{estim1-derivatives} and \eqref{estim2-derivatives}, we can recapture \eqref{1-a} and \eqref{1-b}. The proof will use the same ideas shown in theorem \ref{theorem-nu-zero-1}.

\end{proof}

\begin{example} \label{pro1-ex2}
Let $\{\mu_j\}_{j \in \mathbb{N}}$ be the sequence
\begin{equation*}
\mu_j =
\left\{
\begin{array}{ll}
1/j, \,  \textrm{ if $j$ is odd,}\\[2mm]
 j, \ \ \ \textrm{ if $j$ is even,}
\end{array}
\right.
\end{equation*}
consider the operators $q$ and ${\cal{L}}$ defined by
\begin{equation*}
q(x,D_x) \cdot u = \sum_{j \in \N} u_j \mu_j \varphi_j(x),
\end{equation*}
for each $u = \sum_{j \in \N} u_j  \varphi_j(x) \in \D(\T_x)$, and
\begin{equation*}
{\cal{L}} = D_t + ( a(t) + i b(t)) q(x,D_x), \ (t,x) \in \T^2 = \T_t\times \T_x.
\end{equation*}

Thus, if $a_0$ is an irrational non-Liouville number then the operator ${\cal{L}}$ is (GH). Indeed, by \eqref{alpha-n-L} we obtain
\begin{equation*}
|a_0 \mu_j + \ell | \geqslant
\left\{
\begin{array}{l}
1/j, \ \textrm{ if $j$ is odd,}\\[2mm]
 j^{\, \delta-1}, \  \textrm{ if $j$ is even,}
\end{array}
\right.
\end{equation*}
for each $\ell \in \Z$.
\end{example}

\section{Remarks on time-dependent coefficients}

In this section we are going to take into consideration a natural extension of the separation of variables case. More specifically, we are interested in the study of the following class of  operators
\begin{equation}\label{L-t-depen}
L = D_t + A(t,x,D_x) + i B(t,x,D_x), \quad (t,x)\in \T \times M,
\end{equation}
which satisfy, for each $t \in \T$, the following conditions
\begin{eqnarray}
& A^*(t,x,D_x) = A(t,x,D_x)  \textrm{ and } \ B^*(t,x,D_x) = B(t,x,D_x), \label{hyp-self-1} \\[2mm]
& [A(t,x,D_x),E(x,D_x)] = [B(t,x,D_x),E(x,D_x)]=0, \label{1-hyp3} \\[2mm]
& [A(t,x,D_x),B(t,x,D_x)]=0, \label{1-hyp4}
\end{eqnarray}
where $E$ is an elliptic pseudo-differential operator in $\Psi^m_{el}(M)$.

As an additional hypothesis we assume that $A(t,x,D_x)$ and $B(t,x,D_x)$ are diagonal operators, i.e., for any $u \in \D(\T\times M)$ we have
\begin{align}
A(t,x,D_x) u & = \sum_{j=1}^\infty \sum_{k=1}^{d_j}\alpha_k^j(t)u_k^j(t) e_k^j(x) \nonumber \\[2mm]
             & = \sum_{j\in \N} \left\langle {\mathcal{A}}_j(t) \cdot {U_j}(t), {e^j}(x)  \right\rangle_{\C^{d_j}} \label{A(t)-mult},
\end{align}
and
\begin{align}
B(t,x,D_x) u & = \sum_{j=1}^\infty \sum_{k=1}^{d_j}\beta_k^j(t)u_k^j(t) e_k^j(x) \nonumber \\[2mm]
             & = \sum_{j\in \N} \left\langle {\mathcal{B}}_j(t) \cdot {U_j}(t), {e^j}(x)  \right\rangle_{\C^{d_j}} \label{B(t)-mult},
\end{align}
where
\begin{equation*}
{\mathcal{A}}_j(t) =  \mbox{diag} \left ( \alpha_1^j(t),   \ldots,  \alpha_{d_j}^j(t) \right) \ \ \textrm{ and } \ \
{\mathcal{B}}_j(t) =  \mbox{diag} \left ( \beta_1^j(t),   \ldots,  \beta_{d_j}^j(t) \right),
\end{equation*}
and we are relying on the notation of subsection \ref{reduct-diag}.

We assume that $\alpha^j_k(t), \beta^j_k(t) \in C^\infty (\T;\R)$ satisfies the first order requirement
\begin{equation}\label{alpha-cond}
\sup_{1\leqslant k\leqslant d_j, j\in \N} \left( \lambda_j^{-1/m},\max \left\{ \sup_{t\in \T} |\alpha^j_k (t)|,
                                                        \sup_{t\in \T} |\beta^j_k (t)| \right\}\right)  \leqslant  C,
\end{equation}

\noindent and the continuous action $C^\infty (\T\times M) \mapsto \D(\T \times M)$ satisfies
\begin{equation}\label{beta-cond}
\sup_{1\leqslant k\leqslant d_j, j\in \N} \left( \lambda_j^{-1/m - \omega^j_k(\ell)},\max \left\{ \sup_{t\in \T} |D_t^\ell\alpha^j_k (t)|,
                                                         \sup_{t\in \T} |D_t^\ell\beta^j_k (t)| \right\}\right)  <  +\infty,
\end{equation}

\noindent for some $\omega^j_k(\ell) \in \R$, where $\ell \in \mathbb{N}.$

Thus, the equation $Lu = f$ can be reduced to a sequence of $d_j \times d_j$ linear systems of ordinary differential equations
\begin{equation}\label{time-sys1}
D_t U_j(t) + ({\mathcal{A}}_j(t) + i {\mathcal{B}}_j(t)) U_j(t) = F_j(t),
\end{equation}
or in an equivalent form
\begin{equation}\label{time-sys2}
D_t u_k^j(t) + (\alpha_k^j(t) + i \beta_k^j(t)) u_k^j(t) = f_k^j(t), \ \forall t \in \T,
\end{equation}
for each $k \in \{1, \ldots, d_j\}$.

Thus, since in the time-independent case we can consider the simpler situation where $d_j =1$, for any $j \in \N$, we can also rewrite \eqref{A(t)-mult} and \eqref{B(t)-mult} as
\begin{align}
A(t,x,D_x) u & = \sum_{j=1}^\infty a_j(t)u_j(t) \varphi_j(x), \label{GHgena}\\
B(t,x,D_x) u & = \sum_{j=1}^\infty b_j(t)u_j(t) \varphi_j(x). \label{GHgenb}
\end{align}

Now, by the Weyl formula, conditions \eqref{alpha-cond} and \eqref{beta-cond} are equivalent to
\begin{equation}\label{GHgenD1a}
\max\{ \sup_{t\in \T} |a_j (t)|, \sup_{t\in \T} |b_j (t)| \} \leqslant  C j^{1/n}, \quad j\in \N,
\end{equation}
and
\begin{equation}\label{GHgenD1b}
\sup_{j\in \N} \left( j^{-1/n -r_\ell} \max\{ \sup_{t\in \T} |D_t^\ell a_j (t)|, \sup_{t\in \T} |D_t^\ell b_j (t)| \} \right) < +\infty,
\end{equation}
for some $r_\ell \in \R$, $\ell \in \N$.

\smallskip

\begin{remark}
Clearly, we recaptured the case $a_j(t) = a(t) \mu_j$, $b_j(t) = b(t) \nu_j$, where $r_\ell =0$ for all $\ell\in \N$.

Furthermore, if
\begin{equation*}
L = D_t + (a(t) + ib(t))D_x
\end{equation*}
is a first order differential operator on $\T^2$ (as in J. Hounie \cite{HOU79}, A. Bergamasco \cite{BERG99} and others) then we can take $E(x,D_x) = -\Delta$ on $\T$.

In this situation we have $\sigma_0 = 0$, $d_0=1$, $\sigma_j = j^2$,  $\ds E_j = [e^{-ijx}, \, e^{ijx}]$, $d_j=2$, $j\in \N$ and setting
\begin{equation*}
\psi^j_1 (x) = \varphi_{2j-1}(x) = e^{-ijx} \ \ \textrm{ and } \ \  \psi^j_2 (x) = \varphi_{2j}(x) = e^{ijx},
\end{equation*}
we can rewrite $Lu=f$ as
\begin{equation} \label{GHdoT1}
D_t U_j(t) + \diag ( - a(t) j - ib(t)j, \, a(t) j + ib(t)j) \cdot U_j(t)  =  F_j(t),
\end{equation}
for $t\in \T$ and $j\in \N.$
\end{remark}

\subsection{Reduction to time-independent case}

If $A(t,x,D_x)$ and $B(t,x,D_x)$ are given by \eqref{GHgena} and \eqref{GHgenb}, we define the operators
\begin{align}
A_0(x,D_x) u & = \sum_{j=1}^\infty a_0^ju_j(t) \varphi_j(x), \label{A-geral-Cons}\\
B_0(x,D_x) u & = \sum_{j=1}^\infty b_0^ju_j(t) \varphi_j(x), \label{B-geral-Cons}
\end{align}
where, for each $j \in \mathbb{N},$
\begin{equation}\label{a0-b0-time}
a^j_0 =  (2\pi)^{-1} \int_{0}^{2\pi}a_j(\tau) d\tau \ \ \textrm{ and } \ \  b^j_0 = (2\pi)^{-1} \int_{0}^{2\pi}b_j(\tau) d\tau.
\end{equation}

Now, consider the following sequences of functions in $C^{\infty}(\T)$
\begin{equation*}
\widetilde{A}_j(t) = \int_{0}^{t}a_j(s)ds -a_0^j t, \ \ \widetilde{B}_j(t) = \int_{0}^{t}b_j(s)ds -b_0^j \, t,
\end{equation*}
and, for each for each $u = \sum_{j \in \N} u_j(t) \varphi_j(x) \in \D(\T \times M)$, set
\begin{equation*}
\Psi_{A,B} \cdot u \doteq \sum_{j \in \N} e^{\widetilde{B}_j(t) - i \widetilde{A}_j(t)}u_j(t) \varphi_j(x)
\end{equation*}

\begin{theorem}\label{reduction-time-ind}
Let $\{\tau_j\}$ be the sequence defined by
\begin{equation}\label{max-Bj}
\tau_j = \ds \max_{t \in \T} |b_j(t)|, \ j \in \N.
\end{equation}
and admit that $\{\tau_j\}$ has at most logarithmic growth, i.e.,
   \begin{equation}\label{boundhyp-geral}
                         \limsup_{j \to \infty} \dfrac{\tau_j}{\log (j)} = \tau < + \infty.
   \end{equation}
In these conditions, the following statements hold:
\begin{itemize}
        \item[i.]  $\Psi_{A,B}$ is an isomorphism of the spaces $\D(\T \times M)$ and $C^{\infty}(\T \times M)$;

        \item[ii.] $(\Psi_{A,B})^{-1}\circ L \circ \Psi_{A,B}= L_{A_0,B_0}$ where
\begin{equation}\label{Const-geral}
L_{A_0,B_0} = D_t + A_0(x,D_x) + i B_0(x,D_x);
\end{equation}

  \item[iii.] $L$ is (GH) if, and only if, $L_{A_0,B_0}$ is (GH).
\end{itemize}
\end{theorem}

\begin{proof}
This demonstration uses the same ideas of reduction to normal form shown in subsection \ref{reduct-section}. We point out that, to prove that $\Psi_{A,B}$ is a well defined map from $\D(\T \times M)$ to $\D(\T \times M)$, it is enough to notice that for any $\epsilon>0$ there exist $j_0 \in \N$ such that
\begin{equation*}
\tau_j \leqslant \log(j^{\, \tau + \epsilon}), \ j\geqslant j_0,
\end{equation*}
thus, for any $k \in \N_0$ we obtain $\delta=\delta(k)$ and a positive constant $C$, such that
\begin{align*}
\left|\partial_t^k\left( e^{\widetilde{B}_j(t) - i \widetilde{A}_j(t)} u_j(t) \right)\right| & \leqslant C j^{\delta} e^{\widetilde{B}_j(t)} \\[2mm]
                                                                               & \leqslant C j^{\delta} e^{2 \pi \tau_j} \\[2mm]
                                                                               & \leqslant C j^{\delta + 2 \pi (\tau + \epsilon)}.
\end{align*}
\end{proof}

\begin{example}
Let $\alpha, \beta \in \R$ and define
\begin{eqnarray*}
A(t,x,D_x) u &=& \sum_{j=1}^\infty ( \alpha + \cos(j t)) u_j(t) \varphi_j(x), \ \ \textrm{ and } \\
B(t,x,D_x) u &=& \sum_{j=1}^\infty ( \beta + \sin(j t) ) u_j(t) \varphi_j(x).
\end{eqnarray*}
for each $u = \sum_{j \in \N} u_j(t) \varphi_j(x) \in \D(\T \times M)$.

We have $a_0^j = \alpha$, $b_0^j = \beta$ and $\tau_j \leq 1 + |\beta|$, for all $j \in \N$. Then, $L=A + i B$ is (GH) if, and only if,
\begin{equation*}
L_{A_0,B_0} = D_t + (\alpha + i \beta) \ \mbox{ is (GH).}
\end{equation*}

If $\alpha = \beta = 0$ then $L$ is not (GH) because $L_{A_0,B_0} = D_t$. But, if $|\alpha|+|\beta| \neq 0$ then the operator $L_{A_0,B_0}$ is (GH) and therefore $L$ is (GH).
\end{example}

\

\begin{remark}
In the multidimensional case ($d_j\geqslant 1)$ we set the matrices
\begin{equation*}
\mathcal{A}_j^0 = (2\pi)^{-1}\int_{0}^{2\pi}{\mathcal{A}}_j(t)dt \ \ \textrm{ and }
\ \ \mathcal{B}_j^0 = (2\pi)^{-1}\int_{0}^{2\pi}{\mathcal{B}}_j(t)dt,
\end{equation*}
where ${\mathcal{A}}_j(t)$ and ${\mathcal{B}}_j(t)$ are given in \eqref{A(t)-mult} and \eqref{B(t)-mult}.

Now, we define the functions
\begin{equation*}
\widetilde{{\mathcal{A}}}_j(t) = \int_{0}^{t}{\mathcal{A}}_j(s)ds - t \mathcal{A}_j^0, \ \
\widetilde{{\mathcal{B}}}_j(t) = \int_{0}^{t}{\mathcal{B}}_j(s)ds - t \mathcal{B}_j^0,
\end{equation*}
and for each $u \in \D(\T \times M)$ set

\begin{equation*}
\Psi_{{\mathcal{A}},{\mathcal{B}}} \cdot u \doteq \sum_{j \in \N}
\left\langle e^{\widetilde{{\mathcal{B}}}_j(t) - i \widetilde{{\mathcal{A}}}_j(t)} \cdot {U_j}(t), {e^j}(x)  \right\rangle_{\C^{d_j}}.
\end{equation*}

Next, let
\begin{equation*}
\tau_k^j = \ds \max_{t \in \T} |\beta_k^j(t)|, \ \forall j \in \N \ \textrm{ and } \ k \in \{1, \ldots, d_j\},
\end{equation*}
and define the sequence
\begin{equation*}
\{\tau_j\} \doteq  \ \{ \tau_1^1, \ldots, \tau_1^{d_1}, \tau_2^{1}, \ldots, \tau_2^{d_2}, \ldots, \tau_j^{1}, \ldots,
 \tau_j^{d_j}\}.
\end{equation*}

Thus,   we recapture theorem \ref{reduction-time-ind}.

\end{remark}

\subsection{Global hypoellipticity}

Let $L$ be the operator \eqref{L-t-depen}, as in the beginning of this section, with $A(t,x,D_x)$ and $B(t,x,D_x)$ satisfying \eqref{GHgena} and \eqref{GHgenb}. If $u \in \D(\T \times M)$, an equation $(iL)u = f$ is equivalent to the following sequence of differential equations
\begin{equation}\label{GHgen1}
\partial_t u_j(t) + c_j(t) u_j(t) = f_j(t), t \in \T, \, j\in \N,
\end{equation}

\noindent where $c_j(t) =-b_j(t) + i a_j(t)$.

Let $c^j_0 = -b^j_0 + ia^j_0$, where $a^j_0$ and $b^j_0$ are given in \eqref{a0-b0-time}. Then, for each $j \in \N$ such that $c^j_0 \notin i\Z$ the unique solution of the differential equation \eqref{GHgen1} is
\begin{equation}\label{sol1}
u_j(t)=(1 - e^{-2\pi c^j_0} )^{-1} \int_{0}^{2\pi}{e^{\int_{t}^{t-s} c_j(\tau)d\tau}f_j(t-s)ds},
\end{equation}
or equivalently,
\begin{equation}\label{sol2}
u_j(t) =  (e^{2\pi c^j_0}  - 1 )^{-1} \int_{0}^{2\pi}{e^{\int_{t}^{t+s} c_j(\tau)d\tau}f_j(t + s)ds}.
\end{equation}

As in the separation of variables case, the solution \eqref{sol1} satisfies
\begin{equation}\label{estim1-time}
|\partial^{k}_t u_j(t)|   \leqslant C \ \Theta_j \ j^{-\eta + \delta_k} \ \int_{0}^{2\pi} e^{\int_{t-s}^{t}b_j(\tau)d \tau} ds.
\end{equation}
and \eqref{sol2} satisfies
\begin{equation}\label{estim2-time}
|\partial^{k}_t u_j(t)| \leqslant C \ \Theta_j \ e^{2 \pi b_0^j} \ j^{-\eta + \delta_k} \ \int_{0}^{2\pi}
                               e^{-\int_{t}^{t+s}b_j(\tau)d\tau} ds,
\end{equation}
where $\Theta_j = | 1 - e^{ - 2\pi c^j_0}|^{-1}$, and $\delta_k = k/n+ k r_k$.

\smallskip

For the study of the global regularity of $L$ we split this subsection in two parts: first we are going to propose a generalization of the ``no change of sign condition'', and afterwards, we are going to introduce the notion of ``change of sign condition''.

\subsubsection{General non-change sign condition}

Let $L$ be the operator \eqref{L-t-depen}, where $A(t,x,D_x)$ and $B(t,x,D_x)$ are given in \eqref{GHgena} and \eqref{GHgenb}. We introduce the following notion of the non-change sign condition.

\begin{definition}[GNCS]
We say that the imaginary part $B(t,x,D_x)$ of $L$ satisfies the general non-change sign condition (GNCS) if there exists $j_0 \in \N$ such that $\forall j > j_0$ and the functions $b_j$ do not change sign. In this case we write
\begin{equation*}
\N = \{1, \ldots, j_0 \} \cup {\cal{J}}^- \cup {\cal{J}}^+,
\end{equation*}
where
\begin{equation*}
{\cal{J}}^- =\{j\geqslant j_0; \, b_j(t) \leqslant  0, \ \forall t \in \T\}
\ \ \textrm{ and } \ \
{\cal{J}}^+ =\{j\geqslant j_0; \, b_j(t) \geqslant  0, \ \forall t \in \T\}.
\end{equation*}
\end{definition}

\smallskip

\begin{remark}
Note that if $b_j(t) = b(t) \nu_j$, for some real sequence $\{\nu_j\}_{j \in \N}$ (as in the separation of variables case), then the (GNCS) condition is equivalent to requiring that $b$ does not change sign. For instance, if $b(t)\geqslant 0$, then
\begin{equation*}
{\cal{J}}^- =\{j \in \N; \, \nu_j \leqslant 0\}
\ \ \textrm{ and } \ \
{\cal{J}}^+ =\{j \in \N; \, \nu_j \geqslant 0\}.
\end{equation*}
\end{remark}

Now we introduce some results about the (GH) of $L$.

\begin{theorem}\label{gncs1}
If $B(t,x,D_x)$ satisfies the (GNCS) condition and we have $\lim|b_0^j| = \infty,$ then $L$ is (GH).
\end{theorem}

\begin{proof}
Since $|b_0^j| \to \infty$ then $b_0^j =0$ at most for a finite numbers of indexes $j$, and therefore it is enough to consider the solutions
\eqref{sol1} and \eqref{sol2}.

If, $j \in {\cal{J}}^-$, we will use the expression \eqref{sol1} and if, $j \in {\cal{J}}^+$, we will use the expression \eqref{sol2}. Note that,
\begin{equation} \label{int<0}
j \in {\cal{J}}^- \  \Rightarrow \ e^{\int_{t-s}^{t}b_j(\tau)d \tau} \leqslant 1 \ \ \textrm{ and } \ \
j \in {\cal{J}}^+ \  \Rightarrow \ e^{-\int_{t}^{t+s}b_j(\tau)d \tau} \leqslant 1.
\end{equation}

Now, for $j \in {\cal{J}}^-$ we have $b_0^j \to - \infty$ and hence
\begin{equation*}
\lim_{j \to \infty} \Theta_j = \lim_{j \to \infty} (e^{4\pi b_0^j} - 2e^{2\pi b_0^j}\cos(2\pi a_0^j) + 1)^{-1/2} =1,
\end{equation*}
and for $j \in {\cal{J}}^+$ we have $b_0^j \to + \infty$ and hence
\begin{equation*}
\lim_{j \to \infty} \Theta_j e^{2\pi b_0^j}= \lim_{j \to \infty} (1  - 2e^{-2\pi b_0^j}\cos(2\pi a_0^j) + e^{-4\pi b_0^j})^{-1/2} =1.
\end{equation*}

Thus, by using \eqref{int<0}, the estimates \eqref{estim1-time} and \eqref{estim2-time} becomes
\begin{equation*}
|\partial^{k}_t u_j(t)| \leqslant C \ j^{-\eta + \delta_k},
\end{equation*}
and then $L$ is (GH).

\end{proof}

\begin{corollary}
Admit that $B(t,x,D_x)$ satisfies the (GNCS) condition. If $\{b_0^j\}$ has no subsequence converging to zero, then $L$ is (GH).
\end{corollary}

\begin{proof}
Using the ideas, developed before, in proposition \ref{hy-nu-fraca}, it becomes clear that we obtain the same conclusions of theorem \ref{gncs1} because $0 < C \leqslant |b_0^j|$, for $j\to \infty$.

\end{proof}

As in section \ref{remark-mu}, we are interested in the consequences to the (GH) when zero is an accumulation point of the sequence $\{b_0^j\}$.

We define the set
\begin{equation*}
\Gamma_{A} = \{j \in \N; \, a_0^j \in \Z\},
\end{equation*}
and we consider the following limit
\begin{equation*}
\kappa \doteq\liminf_{j \to \infty} \left( e^{ 4 \pi b_0^j}  -2e^{ 2 \pi b_0^j} \cos(2\pi a_0^j)+ 1\right).
\end{equation*}

Thus, the proof of the following result is a combination of theorems \ref{theorem-nu-zero-2} and \ref{gncs1}.

\begin{theorem}
Admit that $B(t,x,D_x)$  satisfies the (GNCS) condition, zero is an accumulation point of $\{b_0^j\}$ and $b_0^j = 0$ at most for a finite number of indexes $j$. Then:

\begin{enumerate}
        \item [i.] if $\kappa \neq 0$, the operator $L$ is (GH);

        \item [ii.] if $\kappa = 0$, the operator $L$ is (GH) if, and only if, $\Gamma_{A}$ is finite and
\begin{equation}\label{nonLiouville-time}
\inf_{\l \in \Z} |a_0^j + \l| \geqslant C j^{-\delta}, \ \forall j \geqslant R,
\end{equation}
for some $\delta \geqslant 0$, $C>0$ and $R\gg 1$.
\end{enumerate}

\end{theorem}

\begin{remark}
In the separation of variables case we have $a_j(t)=a(t)\mu_j$. Therefore $a_0^j(t)=a_0\mu_j$ and \eqref{nonLiouville-time} is equivalent to saying that $a_0$ is non-Liouville with respect to the sequence $\mu_j.$

\end{remark}

\subsubsection{General change sign condition}

Let $\{b_j(t)\}_{j \in \N}$ be a sequence of non-zero functions which change sign in $C^\infty (\T:\R)$.

By lemma \ref{lamma-chang.sig} we can find the sequences of partitions
\begin{align}
 & t_0^j < \alpha^*_j < \gamma^*_j < t^*_j <\delta^*_j <\beta^*_j < t_0^j +2\pi  \label{ch-sign-max1} \\[2mm]
 & t_0^j < \alpha_*^j < \gamma_*^j < t_*^j <\delta_*^j <\beta_*^j < t_0^j +2\pi  \label{ch-sign-min1}
\end{align}
and for some positive constants $c^*_j, c_*^j$, the following estimates
\begin{align}
& \max_{ t\in [\alpha^*_j, \gamma^*_j]\cup [\delta^*_j,\beta^*_j] } B_{t^*}^j (t)  < -c^*_j  \label{ch-sign-max2}\\[2mm]
& \min_{t\in  [\alpha_*^j, \gamma_*^j]  \cup [\delta_*^j,\beta_*^j] } B_{t_*}^j (t)  >c_*^j  \label{ch-sign-min2}
\end{align}

\smallskip

We introduce an important notion:

\begin{definition}
We say that the sequence $\{b_j(t)\}$ is change-sign-polynomial interval super-log (CSPIL) if there exists $r\geq 0$ such that
\begin{equation}\label{ch-sign-interv1}
\min \{(\gamma^*_j-\alpha^*_j), (\gamma_*^j-\alpha_*^j),  (\beta^*_j-\delta^*_j), (\beta_*^j-\delta_*^j)\}  = O(j^{-r}), \ \ j\to \infty,
\end{equation}
and
\begin{equation}\label{ch-sign-interv2}
\lim_{j\to \infty}\frac{c_j}{\ln (2+j)}=+\infty
\end{equation}
\end{definition}

\begin{remark}
Condition \eqref{ch-sign-interv1} is quite natural since we require that operator $L$ acts continuously on $C^\infty (\T \times M)$. Consider the following examples on $\T\times \T=\T^2$
\begin{align}
& L_1   = D_t + i \cos (t[|D_x|]^N)       \label{ex1-ch-sign}\\[2mm]
& L_2   = D_t + i \cos (t[e^{|D_x|^N}]),  \label{ex2-ch-sign}
\end{align}
for some $N>0$, where $\cos (t \psi (D))$ acts as the multiplier $\cos (t\psi (\xi))$, and $[r]$ stands for the integer part function.

Observe that the operator $L_1$ acts continuously in  $C^\infty (\T \times M)$, for all $N>0$, and  $\cos (t \cos (t[|\xi|]^N) \xi )$
(we can rewrite depending in $j\in \N$ using the Laplace operator as $E(x,D)$) satisfies \eqref{ch-sign-interv1} and \eqref{ch-sign-interv2}.
If $N\leq 1$ the operator is $S^1_{1-N, 0}(\T)$ pseudo-differential operator (see the book of Ruaznasky-Turunen \cite{RT1}) while for $N>1$ it is not pseudo-differential operator but acts continuously in $C^\infty (\T^2)$.

The second operator does not act in $C^\infty (\T^2)$  and $\cos (t[e^{|\xi|]^N}) \xi$ satisfies neither \eqref{ch-sign-interv1} nor \eqref{ch-sign-interv2}.
\end{remark}

\smallskip

\begin{theorem}
If the CSPIL condition is satisfied, then the operator $L$ is not GH.
\end{theorem}

\begin{proof}
Define the functions $\psi_j^*(x) $, $\psi^j_*(x) $, $j\in \N$ as follows:
\begin{align}
\psi_j^*(t) & = \chi_{[ 1/2(\gamma^*_j + \alpha^*_j), 1/2(\delta^*_j + \beta^*_j)]} * \varphi_{\veps j^{-N}} (t),   \label{p-1}\\[2mm]
\psi^j_*(t) & =  \chi_{[  1/2(\gamma_*^j + \alpha_*^j) , 1/2(\gamma_*^j + \beta_*^j)]} * \varphi_{\veps j^{-N}}(t), \label{p-2}
\end{align}
where $0 <\veps \ll 1$ and
\begin{align}
\varphi_\eta & =  \eta^{-1}\varphi (t \eta^{-1}), \quad \eta >0,  \label{p-3} \\[2mm]
\varphi(t)  & =
\left\{
\begin{array}{l} e^{-1/(1-t^2)}, \ \textrm{ if } \ t\in ]-1,1[\\[2mm]
 0, \  \textrm{ if } \ |t| \geq 1.
\end{array}
\right. \label{p-4}
\end{align}

If $0 <\veps \ll 1$ one has
\begin{align}
\psi_j^*(t)  & = 1, \ \ t \in [ \gamma^*_j , \delta_j^* ], \ \ \supp (\psi_j^*) \subset [ \alpha^*_j , \beta_j^* ], \label{p-5}\\[2mm]
\psi^j_*(t)  & = 1, \ \ t \in [ \gamma_*^j , \delta^j_* ], \ \ \supp (\psi^j_*) \subset [ \alpha_*^j , \beta^j_* ]. \label{p-6}
\end{align}

Then $u^*$ and $u_*$ defined by
\begin{align}
u^*_j(t) & = \psi_j^*(t)e^{-iA_{t^*_j}(t) +B_{t^*_j}(t)}  \label{p-7}\\[1mm]
u^*_j(t) & = \psi^j_*(t)e^{iA_{t_*^j}(t) +B_{t_*^j}(t)},  \label{p-8}
\end{align}
satisfy
\begin{align*}
& Lu^* \in C^\infty (\T \times M) \ \ \textrm{ and } \ \ u^* \in D'(\T\times M) \setminus C^\infty (\T \times M), \\[2mm]
& Lu_* \in C^\infty (\T \times M) \ \ \textrm{ and } \ \ u_* \in D'(\T\times M) \setminus C^\infty (\T \times M),
\end{align*}
which yields the desired conclusion.

\end{proof}


\begin{thebibliography}{10}

\bibitem{BDV07}
{\sc Beresnevich, V., Dickinson, D., and Velani, S.}
\newblock Diophantine approximation on planar curves and the distribution of
  rational points.
\newblock {\em Annals of Mathematics 166}, 2 (2007), 367--426.

\bibitem{BERG99}
{\sc Bergamasco, A.~P.}
\newblock Remarks about global analytic hypoellipticity.
\newblock {\em Trans. Amer. Math. Soc. 351}, 10
  (1999), 4113--4126.

\bibitem{BCP04}
{\sc Bergamasco, A.~P., Cordaro, P.~D., and Petronilho, G.}
\newblock Global solvability for a class of complex vector fields on the
  two-torus.
\newblock {\em Communications in Partial Differential Equations 29}, 5-6
  (2004), 785--819.

\bibitem{BDGK15}
{\sc Bergamasco, A.~P., da~Silva, P. L.~D., Gonzalez, R.~B., and Kirilov, A.}
\newblock Global solvability and global hypoellipticity for a class of complex
  vector fields on the 3-torus.
\newblock {\em Journal of Pseudo-Differential Operators and Applications\/} 6
  (2015), 341--360.


\bibitem{BKNZ15}
{\sc Bergamasco, A.~P., Kirilov, A., Zani, S.~L., and Nunes, W. V.~L.}
\newblock Global solutions to involutive systems.
\newblock {\em Proc. Amer. Math. Soc.\/} 143 (2015), 4851--4862.

\bibitem{CH77}
{\sc Cardoso, F., and Hounie, J.}
\newblock Global solvability of an abstract complex.
\newblock {\em Proc. Amer. Math. Soc.}, 65 (1977), 117--124.

\bibitem{CC00}
{\sc Chen, W., and Chi, M.}
\newblock Hypoelliptic vector fields and almost periodic motions on the torus
  tn.
\newblock {\em Communications in Partial Differential Equations 25}, 1-2
  (2000), 337--354.

\bibitem{DR15}
{\sc Delgado, J., and Ruzhansky, M.}
\newblock Fourier multipliers, symbols and nuclearity on compact manifolds.
\newblock {\em arXiv:1404.6479}, to appear in J. Anal. Math..

\bibitem{DR14}
{\sc Delgado, J., and Ruzhansky, M.}
\newblock Kernel and symbol criteria for schatten classes and r-nuclearity on
  compact manifolds.
\newblock {\em Comptes Rendus Mathematique 352}, 10 (2014), 779 -- 784.

\bibitem{DR14JFA}
{\sc Delgado, J., and Ruzhansky, M.}
\newblock Schatten classes on compact manifolds: Kernel conditions.
\newblock {\em Journal of Functional Analysis 267}, 3 (2014), 772 -- 798.

\bibitem{DGY97}
{\sc Dickinson, D., Gramchev, T., and Yoshino, M.}
\newblock First order pseudodifferential operators on the torus: normal forms,
  diophantine phenomena and global hypoellipticity.
\newblock {\em Ann. Univ. Ferrara Sez. VII (N.S.) 41\/} (1997), 51–64.

\bibitem{Fo08}
{\sc Forni, G.}
\newblock On the greenfield-wallach and katok conjectures in dimension three.
  geometric and probabilistic structures in dynamics.
\newblock {\em Contemp. Math 469\/} (2008), 197--213.

\bibitem{GN15}
{\sc Gorodnik, A., and Nevo, A.}
\newblock Quantitative ergodic theorems and their number-theoretic
  applications.
\newblock {\em Bull. Amer. Math. Soc. 52\/} (2015), 65--113.

\bibitem{GPR11}
{\sc Gramchev, T., Pilipovic, S., and Rodino, L.}
\newblock Eigenfunction expansions in $\mathbb{R}^n$.
\newblock {\em Proc. Amer. Math. Soc. 139}, 12 (2011), 4361--4368.

\bibitem{GW1}
{\sc Greenfield, S., and Wallach, N.~R.}
\newblock Global hypoellipticity and liouville numbers.
\newblock {\em Proc. Amer. Math. Soc. 31.1}, 12 (1972), 112--114.

\bibitem{GW3}
{\sc Greenfield, S., and Wallach, N.~R.}
\newblock Globally hypoelliptic vector fields.
\newblock {\em Topology 12}, 3 (1973), 247--253.

\bibitem{GW2}
{\sc Greenfield, S., and Wallach, N.~R.}
\newblock Remarks on global hypoellipticity.
\newblock {\em Trans. Amer. Math. Soc. 183}, 3 (1973), 153--164.

\bibitem{HOU79}
{\sc Hounie, J.}
\newblock Globally hypoelliptic and globally solvable first-order evolution
  equations.
\newblock {\em Trans. Amer. Math. Soc. 252}, 3 (1979), 233--248.

\bibitem{HOU82}
{\sc Hounie, J.}
\newblock Globally hypoelliptic vector fields on compact surfaces.
\newblock {\em Communications in Partial Differential Equations 7}, 4 (1982),
  343--370.

\bibitem{Kat01}
{\sc Katok, A.}
\newblock Cocycles, cohomology and combinatorial constructions in ergodic
  theory, , in collaboration with e. a. robinson, in smooth ergodic theory and
  its applications.
\newblock {\em Proc. Symp. Pure Math. 69\/} (2001), 107--173.

\bibitem{Kat03}
{\sc Katok, A.}
\newblock {\em Combinatorial constructions in ergodic theory and dynamics}.
\newblock American Mathematical Soc., 2003.

\bibitem{Koc09}
{\sc Kocsard, A.}
\newblock Cohomologically rigid vector fields: the katok conjecture in
  dimension 3.
\newblock {\em Annales de l'Institut Henri Poincare (C) Non Linear Analysis
  26}, 4 (2009), 1165--1182.

\bibitem{NT70}
{\sc Nirenberg, L., and Treves, F.}
\newblock On local solvability of linear partial differential equations
  \mbox{part I:} necessary conditions.
\newblock {\em Communications on Pure and Applied Mathematics 23}, 1 (1970),
  1--38.

\bibitem{NT71}
{\sc Nirenberg, L., and Treves, F.}
\newblock Remarks on the solvability of linear equations of evolution.
\newblock {\em Symposia Mathematica, Vol. VII\/} (1971), 325--338.

\bibitem{Petr11}
{\sc Petronilho, G.}
\newblock Global hypoellipticity, global solvability and normal form for a
  class of real vector fields on a torus and application.
\newblock {\em Trans. Amer. Math. Soc. 363}, 12 (2011), 6337--6349.

\bibitem{RT3}
{\sc Ruzhansky, M., and Turunen, V.}
\newblock {\em Pseudo-Differential Operators and Symmetries,
  Pseudo-Differential Operators Theory and Applications}, vol.~2.
\newblock Birkh{\"a}user, Basel, 2010.

\bibitem{RT1}
{\sc Ruzhansky, M., and Turunen, V.}
\newblock Quantization of pseudo-differential operators on the torus.
\newblock {\em Journal of Fourier Analysis and Applications 16}, 6 (2010),
  943--982.

\bibitem{RT2}
{\sc Ruzhansky, M., Turunen, V., and Wirth, J.}
\newblock H\"ormander class of pseudo-differential operators on compact lie
  groups and global hypoellipticity.
\newblock {\em Journal of Fourier Analysis and Applications 20}, 3 (2014),
  476--499.

\bibitem{See65}
{\sc Seeley, R.~T.}
\newblock Integro-differential operators on vector bundles.
\newblock {\em Transactions of the American Mathematical Society 117\/} (1965),
  167--204.

\bibitem{See69}
{\sc Seeley, R.~T.}
\newblock Eigenfunction expansions of analytic functions.
\newblock {\em Proceedings of the American Mathematical Society 21}, 3 (1969),
  734--738.

\bibitem{Shubin}
{\sc Shubin, M.~A.}
\newblock {\em Pseudodifferential operators and spectral theory}, vol.~2.
\newblock Berlin, Springer-Verlag, 2001.

\bibitem{T70}
{\sc Treves, F.}
\newblock Hamiltonian fields, bicharacteristic strips in relation with
  existence and regularity of solutions of linear partial differential
  equations.
\newblock {\em Actes du Congrès International des Mathématiciens
  Gauthier-Villars, Paris\/} (1971), 803--811.

\end{thebibliography}
\end{document}